\documentclass[arxiv]{default}

\title{Roots of $L$-functions of characters over function fields, generic linear independence and biases}
\def\titleHead{Roots of $L$-functions over function fields}

\author{Corentin Perret-Gentil}
\address{Centre de Recherches Mathématiques, Université de Montréal, Canada}
\curraddr{Zürich, Switzerland}
\email{corentin.perretgentil@gmail.com}

\newcommand{\Sec}{\operatorname{Sec}}
\newcommand{\Rel}{\operatorname{Rel}}
\newcommand{\Prim}{\operatorname{Prim}}

\newcommand{\univ}{\text{univ}}

\newcommand{\codd}{,\, \mathrm{odd}}
\newcommand{\Bic}{\mathcal{B}i}
\newcommand{\an}{\text{an}}
\newcommand{\ev}{\text{ev}}
\newcommand{\prim}{\text{prim}}
\newcommand{\Zar}{\mathrm{Zar}}

\ifarxiv
\else
\bibliography{references}
\fi
\usepackage{scrtime}
\date{May 2019. Revised: September 2019}

\subjclass[2010]{14G10, 11T23, 11R58, 11J72, 11N36}
\newcommand{\Bi}{\operatorname{Bi}}

\begin{document}

\begin{abstract}
  We first show joint uniform distribution of values of Kloosterman sums or Birch sums among all extensions of a finite field $\F_q$, for almost all couples of arguments in $\F_q^\times$, as well as lower bounds on differences. Using similar ideas, we then study the biases in the distribution of generalized angles of Gaussian primes over function fields and primes in short intervals over function fields, following recent works of Rudnick--Waxman and Keating--Rudnick, building on cohomological interpretations and determinations of monodromy groups by Katz. Our results are based on generic linear independence of Frobenius eigenvalues of $\ell$-adic representations, that we obtain from integral monodromy information via the strategy of Kowalski, which combines his large sieve for Frobenius with a method of Girstmair. An extension of the large sieve is given to handle wild ramification of sheaves on varieties.
\end{abstract}
\maketitle

\tableofcontents

\section{Introduction and statement of the results}
Throughout, $p$ will denote a prime larger than $5$ and $q$ a power of $p$.

\subsection{Kloosterman and Birch sums}\label{subsec:expSums}

%\subsubsection{Definitions and distribution}\label{subsubsec:expSumsDefDistr}
For an integer $n\ge 1$, and $a\in\F_{q^n}^\times$, we consider the Kloosterman sums
\begin{equation}
  \label{eq:Klr}
  \Kl_{r,q^n}(a)=\frac{1}{q^{n(r-1)/2}}\sum_{\substack{x_1,\dots,x_r\in\F_{q^n}^\times\\ x_1\dots x_r=a}}  e \left(\frac{\tr(x_1+\dots+x_r)}{p}\right)
\end{equation}
of integer rank $r\ge 2$, as well as the Birch sums
\begin{equation}
  \label{eq:Bi}
  \Bi_{q^n}(a)=\frac{1}{q^{n/2}}\sum_{x\in\F_{q^n}^\times} e \left(\frac{\tr(ax+x^3)}{p}\right).
\end{equation}
  Here, we adopt the usual notation $e(z)=\exp(2\pi iz)$ for any $z\in\C$, and $\tr:\F_{q^n}\to\F_p$ is the field trace.\\

  For convenience, let us define the rank of $\Bi_{q^n}$ to be $r=2$, and for $r\ge 2$, we let
  \begin{equation}
    \label{eq:fpn}
    f_{q^n}=\Kl_{r,q^n} \quad (r\ge 2)\qquad\text{or}\qquad\Bi_{q^n} \quad (r=2)
    % f_{p^n}=\begin{cases}
    %   \Kl_{r,p^n}&:r\ge 2\text{, or }\\
    %     \Bi_{p^n}&:r=2
    % \end{cases}
  \end{equation}
  for every integer $n\ge 1$. By the Deligne--Katz equidistribution theorem \cite{KatzGKM} for Kloosterman sums and Livné's work \cite{Liv87} for Birch sums (see also \cite{KatzESDE}), as $q^n\to\infty$ the values
  \[\{f_{q^n}(a): a\in\F_{q^n}^\times\} \ \text{ equidistribute in } \ \Omega_r=\begin{cases}
      [-r,r]\subset\R&:r\text{ even}\\
      \{z\in\C : |z|\le r\}&:r\text{ odd},
    \end{cases}\]
  with respect to the pushforward $\tr_*\mu_r$ of the Haar measure $\mu_r$ on the compact group
  \[G_r(\C),\quad\text{where}\quad G_r:=
    \begin{cases}
      \SU_r&:r\text{ odd}\\
      \USp_r&:r\text{ even}
    \end{cases}
\]
(e.g. the Sato-Tate measure when $r=2$). These statements encompass bounds on $f_{q^n}$ (e.g. Deligne's bound for hyper-Kloosterman sums), and the fact that $f_{q^n}$ is real-valued whenever $r$ is even. Moreover, they can alternatively be phrased as properties of the ``angles'' of Kloosterman and Birch sums, i.e. the
\begin{eqnarray}
  \theta_{1,f,q}(x), \dots, \theta_{r,f,q}(x)\in[0,1],\text{ such that}\nonumber\\
  f_{q^n}(x)=\sum_{i=1}^r e\left(n\theta_{i,f,q}(x)\right)\qquad \text{ for all }n\ge 1, \ x\in\F_q^\times\label{eq:rootsExpSums}
\end{eqnarray}
(whose existence follows from deep work of Grothendieck, Deligne, Katz and others, and will be recalled in due time): they are distributed like the eigenvalues of a Haar-random matrix in $G_r(\C)$.\\

Our first main result is the following generic linear independence statement:
\begin{theorem}[Generic pairwise linear independence]\label{thm:linIndepExp}
  For $r\ge 2$ fixed, let $f$ be as in \eqref{eq:fpn}, and let
  \[E_r:=\dim G_r+\frac{\rank G_r}{2}
    =\begin{cases}
      \frac{2r^2+r-3}{2}&:r\text{ odd}\\
      \frac{2r^2+3r}{4}&:r\text{ even}.
    \end{cases}\]
  For almost all $a,b\in\F_q^\times$, that is for
  \begin{equation}
    \label{eq:almostall}
    (q-1)^2 \left[1+O_{r,p} \left(\frac{\log{q}}{q^{1/(2E_r)}}\right)\right]=(q-1)^2\left(1+o_{r,p}(1)\right)
  \end{equation}
  of them, the angles
  \begin{equation}
    \label{eq:anglesPairwise}
    1, \hspace{0.4cm} \theta_{j,f,q}(a), \hspace{0.4cm} \theta_{j,f,q}(b) \quad\text{with}\quad
\begin{cases}
  1\le j\le r-1&:r\text{ odd}\\
  1\le j\le r/2&:r\text{ even}
\end{cases}
  \end{equation}
  are $\Q$-linearly independent. The implied constants depend only on $r,p$, and only on $r$ in the case of Kloosterman sums.
\end{theorem}

\begin{remark}
  The restriction on $j$ in \eqref{eq:anglesPairwise} is necessary since $\sum_{j=1}^r \theta_{j,f,q}(x)=0$, and if $r$ is even, the angles come by pairs: $\theta_{r/2+j,f,q}(x)=-\theta_{j,f,q}(x)$ ($1\le j\le r/2$).
  %The dependency on $p$ of the implied constants in the case of Birch sums explains why we look at extensions of $\F_q/\F_p$, and not simply at extensions of $\F_p$.
\end{remark}
\begin{remark}
Actually, we will more generally prove Theorem \ref{thm:linIndepExp} for almost all tuples of $t\ge 1$ arguments, when $t=o(\sqrt{\log{q}})$ (e.g. $t$ fixed), with \eqref{eq:almostall} replaced by\footnote{Here and from now on, $\delta_B$ will denote the Kronecker symbol with respect to a binary variable $B$, i.e. $\delta_B=1$ if $B$ is true, $0$ otherwise. In particular, $r^{\delta_{r\text{ odd}}}$ is equal to $r$ if the latter is odd, and to $1$ otherwise.}
\begin{equation}
  \label{eq:almostall2}
  (q-1)^t \left(1+ O_{r,p} \left(\frac{(r^{\delta_{r\text{ odd}}}C)^t\log{q}}{q^{1/(tE_r)}}\right)\right)
\end{equation}
for an absolute constant $C\ge 1$. The implied constants depends again only on $r$ in the case of Kloosterman sums.
\end{remark}

This has several interesting consequences. First, we obtain the \emph{joint distribution} of almost all pairs of values of $f$ in extensions of a fixed base field:

%\subsubsection*{Distribution among values}

\begin{corollary}\label{cor:fUniform}
  For $r\ge 2$, let $f$ be as in \eqref{eq:fpn}, a Kloosterman or Birch sum. For all but 
  $O_{r,p} \left((q-1)^2(\log{q})q^{-1/(2E_r)}\right)$
  couples $a,b\in\F_q^\times$, the random vector
  \[X_{a,b}=\Big(\left(f_{q^n}(a),f_{q^n}(b)\right)\Big)_{1\le n\le N}\]
  \textup{(}with the uniform measure on $[1,N]\cap\N$\textup{)} converges in law as $N\to\infty$ to
  \[\Big(\tr(g_1),\tr(g_2)\Big),\]
  with $g_1,g_2$ independent uniformly distributed in a maximal torus of $G_r(\C)$. Explicitly, $\tr(g_i)$ is distributed like
  \begin{equation}
    \label{eq:explicittrg}
  \begin{cases}
      \sum_{j=1}^{r/2}2\cos(2\pi\theta_j)&:r\text{ even}\\
      \sum_{j=1}^{r-1}e(\theta_j)+e(-\sum_{j=1}^{r-1}\theta_j)&:r\text{ odd},
    \end{cases}  
  \end{equation}
with $\theta_j$ independent uniform in $[0,1]$. Equivalently, the distribution of $\tr(g_i)$ is that of $\tr(h_i^m)$ for any $m\ge r$ and $h_i$ uniform in $G_r(\C)$ with respect to the Haar measure. The implied constant in Landau's notation depends only on $r$ in the case of Kloosterman sums.
\end{corollary}
\begin{figure}[h!]
  \centering
  \begin{equation*}
    \xymatrix@R=0.2cm{
      \F_p^\times\text{ (fixed) }\ar[r]&\F_{q}^\times\ar[r]&\F_{q^n}^\times\ar[r]&\bigcup_{m\ge 1} \F_{q^m}^\times=\overline\F_q^\times\\
    &a,b\ar@{}[u]|-*[@]{\rotatebox{90}{$\in$}}&&}
    \end{equation*}
  %$a,b\in\F_q^\times\subset\F_{q^n}^\times$, \quad $p$ fixed, \quad $q$ arbitrary $p$-power, \quad $n\to\infty$.
  \caption{The asymptotic setting for Section \ref{subsec:expSums}.}
\end{figure}

\begin{remark}
  Applying Deligne's equidistribution theorem and \cite{KatzGKM,KatzESDE} would show that $\Big(f_{q^n}(a+b_1),\dots,f_{q^n}(a+b_t)\Big)_{\substack{a\in\F_{q^n}, \ a+b_i\neq 0}}$ converges in law (with respect to the uniform measure), as $q^n\to\infty$, to a random vector in $\Omega_r^t$ distributed with respect to the product measure $(\tr_*\mu_r)^{\otimes t}$, when $b_i\in\F_{q^n}$ are $t=o(\log(q^n))$ distinct shifts (see e.g. \cite{PG16}, where the dependencies of the errors from \cite{FKMSumProducts} with respect to $t$ are made explicit). However, this only gives information among values that are explicitly related, by fixed shifts.
\end{remark}
\begin{remark}[Discrepancy]
  For the distribution of a single Kloosterman sum of rank 2, conditionally on a linear independence hypothesis, Ahmadi and Shparlinski \cite{AhmSphar10} also obtained bounds on the discrepancy, using lower bounds arising from Baker's theorem. Their results are stated for curves, but the last paragraph of \cite[Section 5.2]{AhmSphar10} explains how they readily extend to Kloosterman sums. Our Theorem \ref{thm:linIndepExp} shows that their discrepancy bounds hold for almost all arguments, and using the same technique, a bound on the discrepancy in Corollary \ref{cor:fUniform} could as well be given.
\end{remark}

Another corollary is the following absence of bias among values of Birch sums and Kloosterman sums in extensions:
\begin{corollary}\label{cor:fBias}
  Let $f_{q^n}$ be either $\Kl_{r,q^n}$ \textup{(}$r\ge 2$ even\textup{)}, $\Bi_{q^n}$, or -- if $r\ge 3$ is odd -- $\Re\Kl_{r,q^n}$ or $\Im\Kl_{r,q^n}$. For all but $O_{r,p} \left((q-1)^2(\log{q})q^{-1/(2E_r)}\right)$
  couples $a,b\in\F_q^\times$, we have
  \begin{eqnarray*}
    \P_{n\le N}\Big(f_{q^n}(a)<f_{q^n}(b)\Big)&:=&\frac{|\{1\le n\le N : f_{q^n}(a)<f_{q^n}(b)\}|}{N}\\
                                                               &\to& 1/2 \quad\text{as }N\to\infty.
  \end{eqnarray*}
  The implied constant in Landau's notation depends only on $r$ in the case of Kloosterman sums.
\end{corollary}
Finally, Theorem \ref{thm:linIndepExp} also yields the following lower bounds, through the method of Bombieri and Katz \cite{BombKatz10}. The first one is not explicit and the value of $n$ is not effective, while the second is weaker but does not suffer from these issues.

\begin{corollary}\label{cor:lowerBounds}
  For $r\ge 2$, let $f$ be as in \eqref{eq:fpn}. For every $\varepsilon>0$ and all but $O_{r,p} \left((q-1)^2(\log{q})q^{-1/(2E_r)}\right)$
  couples $a,b\in\F_q^\times$, we have:
  \begin{enumerate}
  \item\label{item:lowerBounds:subspace} for every $n$ large enough (with respect to $q,r,\varepsilon,a,b$),
    \[|f_{q^n}(a)-f_{q^n}(b)|\ge q^{-\varepsilon n(r-1)}.\]
  \item\label{item:lowerBounds:BW} when $r=2$, for every $n\ge 1$ large enough with respect to $p$,
    \[|f_{q^n}(a)-f_{q^n}(b)|\ge
      (2/\pi^2)\begin{cases}
        q^{-2^{26}3^3\pi p^{3}\log(4p)\log(2n+1/2)}\\
        q^{-C_p\log \left(\frac{n}{e}+\frac{2n+1/2}{q}\right) \frac{\log{q}}{\max(\log{q},2)}}
      \end{cases}\]
    with $C_p=1175\left(5.205+0.946\log{\frac{p-1}{2}}\right)(p-1)^4$.
  \end{enumerate}
\end{corollary}
\begin{remark}
  The second bound in \ref{item:lowerBounds:BW} uses Gouillon's improvement \cite{Gou06} on the Baker--Wüstholz theorem \cite{BakWu93} instead of the latter. The condition on $n$ is only to simplify the expression above: the bound in the proof is fully explicit. Moreover, the first inequality in \ref{item:lowerBounds:BW} is valid for any $n\ge 1$. We can also update the lower bound of Bombieri--Katz \cite[Corollary 4.3(ii)]{BombKatz10} to (assuming $p\ge 5$):
  \[|\Kl_{r,p^n}(a)|\ge (2/\pi)q^{-2C_p\log \left(\frac{n}{e}+\frac{4n+1}{q}\right) \frac{\log{q}}{\max(\log{q},2)}},\]
  with $C_p$ as above.
\end{remark}

\subsection{Angles of Gaussian primes over function fields}\label{subsec:anglesGaussian}

Recently, Rudnick and Waxman \cite{RudWax17} studied refined statistics of angles of Gaussian primes $p=a+ib\in\Z[i]$, after Hecke's equidistribution result and the works that ensued. To give motivation for a conjecture they propose, they develop a function field model where an analogue holds unconditionally.

Explicitly (see \cite[Section 1.3, Section 6]{RudWax17}), consider the quadratic extension $\F_q(S)$ of the function field $\F_p(T)$, $S=\sqrt{-T}$, with the norm $N(f(S))=f(S)f(-S)$. The analogue of the unit circle is
\[\Sb^1_q:=\{u\in\F_q[[S]]^\times : u(0)=1, \  N(u)=1\},\]
and we have a well-defined map $U: \F_q[S]\backslash\{0\}\to\Sb_q^1$, $f\mapsto f/\sqrt{N(f)}$, 
%\[U: \F_q[S]\backslash\{0\}\to\Sb_q^1, \quad f\mapsto f/\sqrt{N(f)},\]
that actually only depends on the ideal $(f)$. For an integer $k\ge 1$, the ``circle'' $\Sb_q^1$ can be divided into $q^\kappa$ sectors ($\kappa=\floor{k/2}$), $\Sec(u,k):=\{v\in\Sb_q^1 : v\equiv u\pmod{S^k}\},$ which are parametrized by
\begin{equation}
  \label{eq:Sk1}
  u\in\Sb_{k,q}^1:=\{u\in R_{k,q}: u(0)=1, \ N(u)=1\}, \quad R_{k,q}:= \left(\F_q[S]/(S^k)\right)^\times.
\end{equation}
Rudnick and Waxman start by showing that if $k\le n$ and
\[N_{k,n}(u):=|\{\pf\normal\F_q[S]\text{ prime} : \deg(\pf)=n, \ U(\pf)\in \Sec(u,k)\}|\]
is the number of primes of fixed degree lying in a sector given by $u\in\Sb_{k,q}^1$, then there is equidistribution in the sectors whenever $\kappa<n/2$:
\[N_{k,n}(u)=\frac{|\{\pf\normal\F_q[S] \text{ prime} : \deg(\pf)=n\}|}{|\Sb_{k,q}^1|}+O \left(q^{n/2}\right)=\frac{q^n/n}{q^\kappa}+O \left(q^{n/2}\right),\]
with an absolute implied constant\footnote{The dependencies of the error with respect to $k$ are not explicit in \cite{RudWax17}, but keeping track of them during the arguments shows that the error in the expression for $N_{k,n}(u)$ above is $O \left(q^{n/2}\kappa/n+\tau(n)^{1/2}q^{n/2}/n\right)$ (recall that we assume that $p\ge 7$), where $\tau$ is the number of divisors function.}. Using a deep result of Katz \cite{Katz16} (based on Deligne's equidistribution theorem and the computation of a monodromy group), they then get an unconditional analogue \cite[Theorem 1.3]{RudWax17} of their conjecture for $\Z[i]$ \cite[Conjecture 1.2]{RudWax17} on the variance of $N_{k,n}$ among all sectors.\\

The notion of Chebyshev bias for primes in arithmetic progressions, studied in depth by Rubinstein and Sarnak \cite{RubSar94}, was extended to function fields by Cha \cite{Cha08}. Further cases of biases in function fields have been considered recently \cite{ChaKim10,ChaFioJou16,ChaFioJou16b,DevMeng18}, particularly in families of curves.

Similarly, one may ask whether there is a bias in the distribution of prime ideals among different sectors as above. To do so, for $u_1,\dots, u_R\in \Sb_{k,q}^1$ distinct, we may look at the $\R^R$-valued random vector
\begin{eqnarray*}
  X_{k,N}(\bs u)&:=&\left(X_{k,N}(u_1),\dots,X_{k,N}(u_R)\right),\text{ where}\\
  X_{k,N}(u_r)&:=&\left(\frac{q^{\kappa}n}{q^{n/2}}\left(N_{k,n}(u_r)-\frac{q^n/n}{q^{\kappa}}\right)\right)_{1\le n\le N}
\end{eqnarray*}
(with the uniform measure on $[1,N]\cap\N$). The normalization is chosen so that $X_{k,N}(u_r)$ is bounded as $N\to\infty$ (with $q,k$ fixed), which will be clear later on.\\

We recall that key inputs in \cite{RubSar94} and \cite{Cha08} to study biases finely are hypotheses about linear independence of roots of $L$-functions, also known as Grand Simplicity Hypotheses (GSH). These are very strong statements and wide open conjectures.

Our second main result is a generic linear independence statement in the setting above, in the same spirit as Theorem \ref{thm:linIndepExp}. It concerns roots
\begin{equation}
  \label{eq:eigenvaluesXi}
  e(\pm\theta_{\Xi,j})\qquad \left(1\le j\le d'(\Xi)\right), \quad \theta_{\Xi,j}\in[0,1],
\end{equation}
of (normalized) $L$-functions associated to characters $\Xi$ of $\Sb_{k,q}^1$ with conductor $3\le d(\Xi)\le 2\kappa-1$, where $d'(\Xi):=(d(\Xi)-1)/2$ (these will be defined more precisely in Section \ref{sec:anglesGP}). The analogue of GSH is:
\begin{hypothesis}\label{hyp:LI}
  The angles $\theta_{\Xi,j}$, for $\Xi\in\widehat\Sb_{k,q}^1$, $1\le j\le d'(\Xi)$, are $\Q$-linearly independent.
\end{hypothesis}
Towards Hypothesis \ref{hyp:LI}, we show:
\begin{theorem}[Generic linear independence]\label{thm:genericLIGP}
  Assume that $p>k$ and let $t=o(\log|\Sb_{k,q}^1|)$ \textup{(}e.g. $t$ fixed\textup{)}. For almost all subsets $S\subset\hat\Sb_{k,q}^1$ of size $t$, that is for
    \[\binom{q^\kappa}{t}\left(1+O_{k,p} \left(\frac{C_{k,p}^t\log{q}}{q^{1/(2t(2\kappa^2-3\kappa+1))}}\right)\right)=\binom{q^\kappa}{t}(1+o_{k,p}(1))\]
  of them, with $C_{k,p}\ge 1$ depending only on $k,p$, the elements
  \[1, \qquad \theta_{\Xi,j} \qquad \left(\Xi\in S, \ 1\le j\le d'(\Xi)\right)\]
  are $\Q$-linearly independent.
  \end{theorem}

  \begin{remark}
    Hypothesis \ref{hyp:LI} would be Theorem \ref{thm:genericLIGP} with $S=\Sb_{k,q}^1$. This is a very strong statement, whose validity may be delicate depending on the relative size of the parameters. Indeed, unlike in the number field situation, there are examples of families of $L$-functions over function fields where linear independence is not satisfied (although with $q$ fixed, and eventually growing genus): see e.g. \cite[Section 6]{Kow08}, \cite[Section 5]{Cha08} and \cite{Li18}.
  \end{remark}
  \begin{remark}\label{rem:weakExplicit1}
  One can get the explicit dependency of the base $C_{k,p}$ with respect to $k,p$ in Theorem \ref{thm:genericLIGP}, at the cost of a weaker error, replacing the latter by $O _{k,p}\left(\frac{(C(k+1)^{k+1})^t\log\log{q}}{\log{q}}\right)$ with $C$ absolute. Under a group theoretic conjecture, one could do so while keeping the strength of Theorem \ref{thm:genericLIGP}: see Remark \ref{rem:BurnsideImprove}.
\end{remark}

  Let us now explain how this relates to biases and the random vectors $X_{k,N}(\bs u)$ defined above. We adapt classical arguments \cite{RubSar94,MarNg17,Dev18} to the function field setting, as in \cite{Cha08,DevMeng18}, to show:
  
\begin{theorem}[Limiting distribution, expected value]\label{thm:limitDistrExp}
  The random vector $X_{k,N}(\bs u)$ admits a compactly supported limiting distribution as $N\to\infty$ with $\kappa <N/2$ fixed. Namely, it converges in law to a $\R^R$-valued random variable $X_{k}(\bs u)$. Moreover, the expected value of the latter is
  \[\E \left(X_{k}(\bs u)\right)=\Big(-|\{b\in\Sb_{k,q}^1 : b^2=u_r\}|/2\Big)_{1\le r\le R}\subset \{-1/2,0\}^R,\]
  which means that there should be a bias towards sectors parametrized by non-squares.
\end{theorem}

\begin{theorem}[Continuity, symmetry, bias]\label{thm:distrPropGPLI}
  If Hypothesis \ref{hyp:LI} holds and $R<\frac{\kappa-1}{2}$ is an integer, the distribution of $X_{k}(\bs u)$ is:
  \begin{enumerate}
  \item\label{item:thm:distrPropGPLIABsCont} absolutely continuous: there exists a Lebesgue integrable function $f$ on $\R^R$ such that $\P(X_{k}(\bs u)\in A)=\int_A fd\bs x$ for all Borel subsets $A\subset\R^R$.
  \item\label{item:thm:distrPropGPLISymmetry} symmetric with exchangeable components around its mean:\\
    for $X_k^0(\bs u):=X_k(\bs u)-\E(X_k(\bs u))$, we have
    \[X_k^0(\bs u)\sim -X_k^0(\bs u), \ \sigma(X_k^0(\bs u))\]
    for any permutation $\sigma\in \Sf_R$ of the coordinates.
  \end{enumerate}
  Hence,
  \[\lim_{N\to\infty}\P\Big(X_{k,N}(u_1)<\dots<X_{k,N}(u_R)\Big)=\P \Big(X_{k}(\bs u)_1<\dots<X_{k}(\bs u)_R\Big),\]
  which is $1/R!$ if the $u_i$ are all squares or all non-squares. If $u_2$ is a square while $u_1$ is not, and $\kappa>5$, then $\lim_{N\to\infty} \P\big(X_{k,N}(u_1)<X_{k,N}(u_2)\big)<1/2$.
\end{theorem}
\begin{remark}
  The restriction $R<\frac{\kappa-1}{2}$, rather strong with respect to the maximum $R=q^\kappa$, comes from the fact that the $L$-functions have finitely many zeros, in contrast with the number field case.
\end{remark}

Hence, our generic linear independence statement, Theorem \ref{thm:genericLIGP}, implies the following towards an unconditional Theorem \ref{thm:distrPropGPLI}:
\begin{corollary}[of Theorem \ref{thm:genericLIGP}]\label{cor:distrPropGPweak}
  Assuming that $p>k$, the limiting distribution $X_k(\bs u)$ of Theorem \ref{thm:limitDistrExp} is:
  \begin{enumerate}
  \item continuous: $\P(X_k(\bs u)=\bs a)=0$ for any $\bs a\in\R^R$. In particular, for $u\in\Sb_{k,q}^1$, $\lim_{N\to\infty} \P(X_{k,N}(u)>0)=\P(X_{k}(u)>0).$
  \item a pushforward of the Lebesgue measure on a torus of dimension
      \[\gg_{\varepsilon,k}(\log|\Sb_{k,q}^1|)^{1-\varepsilon},\quad\text{ for any }\varepsilon>0.\]
  \end{enumerate}
\end{corollary}
\begin{remark}
  Concerning the stronger properties of Theorem \ref{thm:distrPropGPLI} (absolute continuity, symmetry), Devin \cite{Dev18} and Martin--Ng \cite{MarNg17} have shown that they hold under weaker conditions than full linear independence. However, we cannot exploit these here since their statements always involve all the roots/eigenvalues, while results obtained from the large sieve will be limited to a small subset.
\end{remark}

\subsection{Prime polynomials in short intervals}\label{subsec:PPSI}

Some of the techniques in \cite{RudWax17} actually originate from Keating and Rudnick \cite{KeatRud14}, who showed function field analogues of a conditional result of Goldston--Montgomery on primes in short intervals and of a conjecture of Hooley on the variance of primes in arithmetic progressions with fixed modulus.

For $A\in\F_q[T]$ of degree $n\ge 1$ and $1\le h\le n$,
\[\nu_h(A):=\sum_{\substack{f\in\F_q[T] \\ \deg(f-A)\le h}} \Lambda(f)\]
counts prime polynomials in a ``short interval'' around $A$, weighted by the function field von Mangoldt function $\Lambda$ (defined by $\Lambda(f)=\deg(P)$ if $f=P^k$, $P\in\F_q[T]$ prime, $\Lambda(f)=0$ otherwise). The mean value over the centers $A$ having degree $n$ is
\begin{equation}
  \label{eq:varnu}
  \E_{q,n} \left(\nu_h\right):=\frac{1}{q^n}\sum_{\substack{A\in\F_q[T]\text{ monic}\\ \deg(A)=n}} \nu_h(A)=q^{h+1}\left(1-\frac{1}{q^n}\right)
\end{equation}
(see \cite[(2.7), Lemma 4.3]{KeatRud14}). Keating and Rudnick, \cite[Theorem 2.1]{KeatRud14}, using another equidistribution result of Katz \cite{Kat13} when $h<n-3$, compute the corresponding variance explicitly, obtaining an unconditional analogue of the Goldston--Montgomery result mentioned above.

% One can then study the variance
% \[\Var_{q,n} \left(\nu_h\right):=\frac{1}{q^n}\sum_{\substack{A\in\F_q[T]\text{ monic}\\ \deg(A)=n}}\left|\nu_h(A)-\E_{q,n} \left(\nu_h\right)\right|^2.\]
% Keating and Rudnick show \cite[Theorem 2.1]{KeatRud14}, using another equidistribution result of Katz \cite{Kat13}, that if $h<n-3$, then $\lim_{q\to\infty} \Var_{q,n} \left(\nu_h\right)q^{-h-1}=n-h-2$.
% This is an unconditional analogue of the Goldston--Montgomery result mentioned above.

Any monic $A\in\F_q[T]$ of degree $n$ can be written uniquely as
\[A=T^{h+1}B+C\quad\text{with}\quad
  \begin{cases}
      B\text{ monic, }\deg(B)=n-h-1\\
      \deg(C)\le h,
  \end{cases}\]
and $\nu_h(A)=\nu_h(T^{h+1}B)$ only depends on $B$. This observation allows us to fix $n-h=:m$ and take $n\to\infty$. For $B_1,\dots,B_R\in\F_q[T]$ distinct and monic of degree $m-1$, we can study the $\R^R$-valued random vector of biases
\begin{eqnarray*}
  X_{m,N}(\bs B)&:=&\big(X_{m,N}(B_1),\dots,X_{m,N}(B_R)\big),\text{ where}\\
  X_{m,N}(B_r)&:=&\left(\frac{q^{m}}{q^{n/2+1}}\Big(\nu_{n-m}(T^{n-m+1}B_r)-\E_{q,n}(\nu_{n-m})\Big)\right)_{1\le n\le N}
\end{eqnarray*}
(with the uniform measure on $[1,N]\cap\N$), the expected values being those in \eqref{eq:varnu}. Again, the normalization is chosen so that $X_{m,N}(u_r)$ is bounded as $N\to\infty$ (with $q,m$ fixed), which will be clear later on.\\

 In this setting, we obtain results analogous to those exposed in Section \ref{subsec:anglesGaussian}. Let
\begin{equation}
  \label{eq:eigenvalueschi}
  e \left(\theta_{\chi,j}\right) \quad (1\le j\le d-1), \quad \theta_{\chi,j}\in[0,1],
\end{equation}
be the roots associated to the $L$-function associated to an even Dirichlet character $\chi$ modulo $T^m\in\F_q[T]$ (see Section \ref{sec:anglesGP} for the precise definitions), for $2\le d\le m$.

\begin{hypothesis}\label{hyp:LIchi}
  The angles $\theta_{\chi,j}$, for $\chi\ (\text{mod }T^{m})\text{ even}$, $1\le j\le \cond(\chi)-2$, are $\Q$-linearly independent.
\end{hypothesis}

\begin{theorem}[Generic linear independence]\label{thm:genericLIGPchi}
  Assume that $m$ is odd, $p>m$ and $t=o(\log(q^{m-1}))$ \textup{(}e.g. $t$ fixed\textup{)}. For almost all subsets $S$ of size $t$ of even Dirichlet characters mod $T^{m}$, that is for
  \[\binom{q^{m-1}}{t}\left(1+O_{p,m} \left(\frac{C_{m,p}^t\log{q}}{q^{1/(2t(m-2)^2)}}\right)\right)=\binom{q^{m-1}}{t}(1+o_{p,m}(1))\]
  of them, with $C_{m,p}\ge 1$ depending only on $p,m$, the elements
  \[1, \qquad \theta_{\chi,j} \qquad \left(\chi\in S, \ 1\le j\le \cond(\chi)-2\right)\]
  are $\Q$-linearly independent.
\end{theorem}

\begin{theorem}[Limiting distribution, expected value]\label{thm:limitDistrExpVar}
  The random vector $X_{m,N}(\bs B)$ admits a compactly supported limiting distribution as $N\to\infty$ with $m>3$ fixed. Namely, it converges in law to a $\R^R$-valued random variable $X_{m}(\bs B)$. Moreover, the latter has mean zero.
\end{theorem}
\begin{remark}
  There is no bias here, unlike in Theorem \ref{thm:limitDistrExp}, simply because the von Mangoldt weight was kept.
\end{remark}

\begin{theorem}[Continuity, symmetry]\label{thm:distrPropGPLIchi}
  If Hypothesis \ref{hyp:LIchi} holds and $R<m/2-1$, the distribution of $X_m(\bs B)$ is absolutely continuous, and symmetric with exchangeable components. In particular,
  \[\lim_{N\to\infty}\P\Big(X_{m,N}(B_1)<\dots<X_{m,N}(B_R)\Big)=\frac{1}{R!}.\]
\end{theorem}
Towards an unconditional Theorem \ref{thm:distrPropGPLIchi}, we obtain:
\begin{corollary}[of Theorem \ref{thm:limitDistrExpVar}]\label{cor:distrPropGPweakchi}
  Assuming $m$ odd and $p>m$, the limiting distribution $X_m(\bs B)$ from Theorem \ref{thm:limitDistrExpVar} is:
  \begin{enumerate}
  \item continuous: $\P(X_m(\bs B)=\bs a)=0$ for any $\bs a\in\R^R$. In particular, for $B\in\F_q[T]$ of degree $m-1$,
    \[\lim_{N\to\infty} \P(X_{m,N}(B)>0)=\P(X_{m}(B)>0).\]
  \item a pushforward of the Lebesgue measure on a torus of dimension
      \[\gg_{\varepsilon,m}\left(\log(q^{m-1})\right)^{1-\varepsilon},\quad\text{ for any }\varepsilon>0.\]
  \end{enumerate}
\end{corollary}
\begin{remark}
  The assumption that $m$ is odd is technical, to get the integral monodromy in Theorem \ref{thm:monEvenDir}. It is anyway mild, since if $m$ is even, one may as well look at shorter intervals of odd size $m-1$.
\end{remark}
\begin{remark}\label{rem:weakExplicit2}
  Again, if one wants explicit dependency of $m,p$ in the base of $t$ in Theorem \ref{thm:genericLIGPchi}, at the price of a weaker error, one may replace the latter by $O_{p,m} \left(\frac{(C(m+1)^{m+3})^t\log\log{q}}{\log{q}}\right)$ with $C$ absolute.
\end{remark}

% For $Q\in\F_q[T]$ of degree $\ge 1$, $A\in\F_q[T]$ coprime to $Q$ and $n\ge 1$, let
% \[\Psi(n,Q,A)=\sum_{\substack{f\in\F_q[T] \ \text{monic} \\ \deg(f)=n\\ f\equiv A\pmod{Q}}}\Lambda(f).\]
% By the prime polynomial in arithmetic progressions, $\Psi(n,Q,A)\sim q^n/\phi(Q)$ as $n\to\infty$, where $\phi(Q)=|(\F_q[T]/(Q))^\times|$. Keating and Rudnick \cite{KeatRud14} studied the the variance
% \[\Var(\Psi(n,Q,\cdot)):=\sum_{\substack{A\in(\F_q[T]/(Q))^\times\\ \mathrm{monic}}} \left|\Psi(n,Q,A)-\frac{q^n}{\varphi(Q)}\right|^2\]
% of $\Psi$ over the classes $A$. They get an asymptotic expression depending on whether $n<\deg(Q)$ or not (see \cite[Theorem 2.2]{KeatRud14}).

\subsection{Outline of the strategy, previous works, and organization of the paper}\label{subsec:methods}

The existence and properties of the limiting distribution under linear independence hypotheses (Theorems \ref{thm:limitDistrExp} \ref{thm:distrPropGPLI}, \ref{thm:limitDistrExpVar} and \ref{thm:distrPropGPLIchi}) follow the methods developed in \cite{RubSar94,Cha08,MarNg17}. The continuity statement in Corollary \ref{cor:distrPropGPweak}, under weaker results than full linear independence, is obtained through an idea of Devin \cite{Dev18,DevMeng18}.

The main results are then Theorems \ref{thm:linIndepExp}, \ref{thm:genericLIGP} and \ref{thm:genericLIGPchi} on generic linear independence. Combining his large sieve for Frobenius over finite fields \cite{KowLS06,KowLargeSieve08} with a method of Girstmair \cite{Girst82,Girst99}, Kowalski proved \cite{Kow08} that a linear independence condition holds generically in some families of $L$-functions of curves over finite fields. This was recently extended by Cha, Fiorilli and Jouve \cite{ChaFioJou16b} to certain families of elliptic curves over function fields, where the underlying symmetry is orthogonal instead of being symplectic.

We use similar ideas to prove Theorems \ref{thm:linIndepExp} and \ref{thm:genericLIGP}, with the families of curves replaced by families of exponential sums or characters. More precisely, by work of Deligne and Katz \cite{DelEC,Katz16}, there are families of $\ell$-adic sheaves on $\G_m$ (resp. on a variety parametrizing primitive characters $\Xi$ or $\chi$ as above) such that the (reversed) characteristic polynomial of the Frobenius acting on a stalk yields the roots (resp. $L$-function) of the corresponding exponential sums (resp. characters).

Unlike in \cite{Kow08,ChaFioJou16b}, these are not sheaves of $\Z_\ell$-modules, but of $\Oc_\lambda$-modules, for $\lambda$ a valuation on the ring of integers $\Oc$ of a number field. In the work of Kowalski and Cha--Fiorilli--Jouve, the monodromy structure is symplectic or orthogonal (the latter being the source of complications handled by Jouve); here, it is either special linear, symplectic or projective general linear.

Another difficulty arises in bounding sums of Betti numbers appearing in the large sieve for Frobenius, because certain sheaves are not defined on curves nor have tame ramification, as assumed by Kowalski and Cha--Fiorilli--Jouve. This yields Theorem \ref{thm:largeSieve}, and answers in this case a question of Kowalski (\cite[Remark 4.8]{KowLS06}).

To apply this variant of the large sieve for Frobenius, we also need information on integral monodromy groups of the sheaves, whereas only information about the monodromy groups over $\C$ (i.e. after taking a Zariski closure) is a priori available from Katz's work \cite{KatzGKM,KatzESDE,Kat13,Katz16}. This is overcome using deep results of Larsen and Pink through ideas of Katz (or more precise results in the case of Kloosterman sums). Unlike in \cite{ChaFioJou16b}, strong approximation for arithmetic groups cannot be used.
\begin{remark}[Frobenius tori]
  As explained in \cite[Section 7]{Kow08}, another way to get generic linear independence results is by applying an effective version of Chebotarev's density theorem with Serre's theory of Frobenius tori. However, as explained in \cite[p. 54]{Kow08}, controlling the uniformity with respect to the size of the subsets/tuples considered (crucial for the questions we consider) is more subtle.
\end{remark}
\begin{remark}[Prime polynomials in arithmetic progressions]
  In \cite{KeatRud14}, Keating and Rudnick also study the variance of prime polynomials in arithmetic progressions, and get as well an asymptotic expression (see \cite[Theorem 2.2]{KeatRud14}). In one of the ranges, this uses another equidistribution result of Katz \cite{Kat13Primitive}. The latter is more complicated, relying on the ideas developed in \cite{KatzMellin12}, because the family involved is not parametrized by an algebraic variety. While results similar to those of Section \ref{subsec:PPSI} could probably be obtained (see also \cite{Cha08}), we leave that to future work for this reason. 
\end{remark}

In Sections \ref{sec:KlBi}, \ref{sec:anglesGP} and \ref{sec:variancePPSI}, respectively for Kloosterman/Birch sums, Gaussian prime polynomials, and prime polynomials in short intervals, we:
\begin{enumerate}
\item Give the cohomological interpretations due to Katz, which gives rise to the eigenvalues from \eqref{eq:rootsExpSums}, \eqref{eq:eigenvaluesXi} and \eqref{eq:eigenvalueschi} respectively.
\item For Gaussian prime polynomials and prime polynomials in short intervals:
  \begin{enumerate}
  \item Show the existence of the limiting distributions (Theorems \ref{thm:limitDistrExp} and \ref{thm:limitDistrExpVar}).
  \item Prove the additional properties of the distributions under Hypotheses \ref{hyp:LI} and \ref{hyp:LIchi} (Theorems \ref{thm:distrPropGPLI} and \ref{thm:distrPropGPLIchi}).
  \end{enumerate}
\item Prove Corollaries \ref{cor:fUniform} and \ref{cor:lowerBounds}, \ref{cor:fBias} and Corollaries \ref{cor:distrPropGPweak}, \ref{cor:distrPropGPweakchi}, from the generic linear independence Theorems \ref{thm:linIndepExp}, \ref{thm:genericLIGP} and \ref{thm:genericLIGPchi} respectively.
\end{enumerate}
Finally, Sections \ref{sec:largeSieve}, \ref{sec:genericMax} and \ref{sec:proofGenericLI} are dedicated to proving these generic linear independence statements.

\subsection{Notations}
For a prime $p\ge 7$ and a field $E$ with ring of integers $\Oc$, we let $\Spec_{1}(\Oc)$ (resp. $\Spec_{p}(\Oc)$) be the set of all non-zero prime ideals (equivalently, valuations on $\Oc$) having degree $1$ (resp. not lying above $p$), and $\Spec_{1,p}(\Oc)=\Spec_1(\Oc)\cap\Spec_p(\Oc)$. If $\lambda\in\Spec_{1,p}(\Oc)$, we denote by $E_\lambda, \Oc_\lambda$ the completions, and $\F_\lambda\cong\Oc/\lambda$ the residue field. Note that $\F_\lambda\cong\F_\ell$, where $\ell$ is the prime above which $\lambda$ lies.
\subsection{Acknowledgements}
The author thanks Lucile Devin, Michele Fornea, Javier Fresán, Florent Jouve and Will Sawin for helpful discussions and comments. Will Sawin in particular provided a better way to bound the sums of Betti numbers in the large sieve, leading to stronger results; the idea and proof of Theorem \ref{thm:largeSieve}\ref{item:compSystem} are due to him. We thank the organizers of the 2019 Shaoul fund IAS Function field arithmetic workshop in Tel-Aviv for providing the opportunity for some of these exchanges. We are grateful to the anonymous referees who provided helpful and detailed comments to improve the manuscript. The author was partially supported by Koukoulopoulos' Discovery Grant 435272-2013 of the Natural Sciences and Engineering Research Council of Canada, and by Radziwiłł's NSERC DG grant and the CRC program.

\section{Kloosterman sums and Birch sums}\label{sec:KlBi}

\subsection{Cohomological interpretation}
\begin{theorem}[Deligne, Katz]\label{thm:Klconstr}
  Let $E=\Q(\zeta_{4p})$, with ring of integers $\Oc$. For every $\lambda\in\Lambda:=\Spec_{p}(\Oc)$, there exists:
  \begin{enumerate}
  \item for every integer $r\ge 2$, a lisse sheaf $\Klc_{r,\lambda}$ on $\G_{m,\F_p}$ of free $\Oc_\lambda$-modules, of rank $r$, pure of weight $0$, such that for every finite field $\F_q$ of characteristic $p$ and $x\in\F_q^\times$,
  \[\tr \left(\Frob_{\F_q} \mid \left(\Klc_{r,\lambda}\right)_{x}\right)=\Kl_{r,q}(x),\]
  the normalized hyper-Kloosterman sum of rank $r$ defined in \eqref{eq:Klr}. Moreover, the family $(\Klc_{r,\lambda})_{\lambda\in\Lambda}$ forms a compatible system\footnote{We recall that this means that for every $\lambda\in\Lambda$, every finite field $\F_q$ of characteristic $p$ and every $x\in\F_q^\times$, the reverse characteristic polynomial $\det(1-T\Frob_{\F_q}\mid \left(\Klc_{r,\lambda}\right)_{x})\in\Oc_\lambda[T]$ has coefficients in $E$ that moreover do not depend on $\lambda$; see \cite[Section II]{Katz01}.}.
\item\label{item:Biconstr} a lisse sheaf $\Bic_\lambda$ on $\G_{m,\F_p}$ of free $\Oc_\lambda$-modules, of rank $2$, pure of weight $0$, such that for every field $\F_q$ of characteristic $p$ and $x\in\F_q^\times$,
  \[\tr \left(\Frob_{\F_q} \mid \left(\Bic_{\lambda}\right)_{x}\right)=\Bi_{q}(x),\]
  the normalized Birch sum defined in \eqref{eq:Bi}. Moreover, the family $(\Bic_{\lambda})_{\lambda\in\Lambda}$ forms a compatible system.
  \end{enumerate}
\end{theorem}
\begin{proof}
  \begin{enumerate}
  \item This is \cite[Theorem 4.1.1/Section 8.9]{KatzGKM}. To normalize by a Tate twist, we enlarge the ring of definition to $\Z[\zeta_{4p}]$, which is enough since $\sqrt{p}\in\Z[\zeta_{4p}]$ by the evaluation of quadratic Gauss sums (see \cite[11.0]{KatzGKM}).
  \item This is contained in \cite[7.12]{KatzESDE} (see also \cite[Part 3]{KatzMonodromyFamES}), along with \cite{KatzGKM} for the definition over $\Oc_\lambda$ of the $\ell$-adic Fourier transform.
  \end{enumerate}  
\end{proof}

The roots of the characteristic polynomial of $\Frob_{\F_q}$ acting on the stalks at $x\in\F_q^\times$ of any of the sheaves in the system $(\Klc_{r,\lambda})_{\lambda\in\Lambda}$, resp. $(\Bic_{\lambda})_{\lambda\in\Lambda}$, are then the $e(\theta_{i,f,q}(x))\in\C$ ($1\le i\le r$) giving \eqref{eq:rootsExpSums}, when $f=\Kl_{r,q}$, resp. $f=\Bi_{q}$ ($r=2$).\\

We now prove the three corollaries of Theorem \ref{thm:linIndepExp} (generic linear independence of the roots) stated in Section \ref{subsec:expSums}.

\subsection{Joint uniform distribution: Corollaries \ref{cor:fUniform} and \ref{cor:fBias}} 

  Let $a,b\in\F_q^\times$ be such that the conclusion of Theorem \ref{thm:linIndepExp} holds. By the Kronecker--Weyl equidistribution theorem (see e.g. \cite[Section 4.1]{Dev18} or \cite[Appendix B]{MarNg17}), the random vector
  \[\Big(n\theta_{i,f,q}(a), \ n\theta_{j,f,q}(b) : 1\le i,j\le r\Big)_{n\le N}\]
  equidistributes in $[0,1]^{2 r}$ as $N\to\infty$. It follows at once by \eqref{eq:rootsExpSums} that $X_{a,b}$ converges in law to a pair of independent random variables distributed like \eqref{eq:explicittrg} as $N\to\infty$.

  Finally, the equivalence of the distribution of \eqref{eq:explicittrg} and traces of large enough powers of matrices in $G_r(\C)$ is the content of \cite[Theorem 2.1]{Rains97}.

  Corollary \ref{cor:fBias} is then an immediate consequence, by applying the portmanteau theorem to the random variable $\tr(g_1)-\tr(g_2)$ (or its real/imaginary parts), which is symmetric around its mean 0. \qed

  \subsection{Lower bounds: Corollary \ref{cor:lowerBounds}}

  We follow the method of Bombieri and Katz \cite[Sections 3--4]{BombKatz10}, based on the subspace theorem from \cite{Ev84,vdpS91} and the Baker--Wüstholz theorem \cite{BakWu93}.

Let $a,b\in\F_q^\times$ be such that the conclusion of Theorem \ref{thm:linIndepExp} holds. By \eqref{eq:rootsExpSums}, we have
\[F(n):=f_{q^n}(a)-f_{q^n}(b)=\sum_{i=1}^r \big(e\left(n\theta_{i,f,q}(a)\right)-e\left(n\theta_{i,f,q}(b)\right)\big).\]
The Skolem--Mahler--Lech theorem (see \cite[Theorem 2.1(i)]{BombKatz10}) shows that if none of
\[e \left(\frac{\theta_{i,f,q}(x)}{\theta_{j,f,q}(x)}\right) \ \ (x\in\{a,b\}, 1\le i<j\le r),\quad e \left(\frac{\theta_{i,f,q}(a)}{\theta_{j,f,q}(b)}\right) \ \ (1\le i,j\le r)\]
are roots of unity, which holds by linear independence, then there are only finitely many $n$ (with $a,b,r,q$ fixed) such that $F(n)=0$.

The subspace theorem \cite{Ev84,vdpS91} (see \cite[Theorem 3.1]{BombKatz10}) shows that, after multiplying by $q^{n \frac{r-1}{2}}$ (i.e. de-normalizing), for every $n\ge 1$ large enough (with respect to the roots $\theta_{i,f,q}$, i.e. with respect to $a,b,r,q,\varepsilon$), either $F(n)=0$, or $F(n)$ satisfies the lower bound of Corollary \ref{cor:lowerBounds}\ref{item:lowerBounds:subspace}. With the above, this proves the first part of the corollary.\\

For the second part, we assume that $r=2$. For any integers $k_0,k_1\in\Z$ and $\theta_0,\theta_1\in[0,1]$, we have
\begin{eqnarray*}
  &&|\cos(2n\pi\theta_0)-\cos(2n\pi\theta_1)|=2|\sin(n\pi(\theta_0+\theta_1))\sin(n\pi(\theta_0-\theta_1))|\\
  &&=2\prod_{j=0}^1|\sin(n\pi\tau_j-k_j\pi)|, \qquad \tau_j=\theta_0+(-1)^{j}\theta_1\\
  &&\ge 2\prod_{j=0}^1 \frac{2\left|n\pi\tau_j-k_j\pi\right|}{\pi}=\frac{2}{\pi^2}\prod_{j=0}^1\left|n\log(e(\tau_j))-k_j\log(-1) \right|,
\end{eqnarray*}
where the inequality holds if $k_j$ is chosen to minimize $|n\tau_j-k_j|$.

We can now apply the Baker--Wüstholz theorem \cite[Theorem, p. 20]{BakWu93} as in \cite[Section 4]{BombKatz10}, or its improvement with respect to the numerical constants by Gouillon \cite{Gou06}, giving the first and second expressions in Corollary \ref{cor:lowerBounds}\ref{item:lowerBounds:BW}. As the arguments are essentially the same, we only give the second one. If $1$, $\theta_0$, $\theta_1$ are linearly independent, then \cite[Corollary 2.2]{Gou06} shows that this is
\begin{eqnarray}\label{eq:BWineq1}
  \ge \frac{2}{\pi^2}\prod_{j=0}^1 \exp \left(-9400 \left(3.317+\frac{1.888}{d}+0.946\log{d}\right)d^4h_jA_j\right),
\end{eqnarray}
where $A_j$ is any real number satisfying $\log A_j\ge\max (1, h(e(\tau_j)), |\tau_j|/d, 1/d)$,
\begin{eqnarray*}
  h_j&=&\max \left(\log \left(\frac{n}{ed}+\frac{k_j}{dA_1}\right),\frac{1000}{d},498+\frac{284}{d}+142\log{d}\right),\\
  d&=&[\Q(e(\tau_0),e(\tau_1)):\Q]/2,
\end{eqnarray*}
for $h_0$ the absolute logarithmic Weil height. We have $h_0(e(\tau_j))\le h_0(e(\theta_0))+h_0(e(\theta_1))$.

Let us now assume that $(\theta_0,\theta_1)=(\theta_{i,f,q}(a),\theta_{i,f,q}(b))$ are moreover angles of exponential sums \eqref{eq:rootsExpSums}. Then $q^{1/2}e(\pm\theta_j)$ is an algebraic integer, so $h_0(e(\tau_j))\le \log{q}$. Regarding the degree, we have $1\le d\le (p-1)/2$ as in \cite[Proof of Corollary 4.3]{BombKatz10}, because Kloosterman/Birch sums are sums of $p$th roots of unity. Thus, we may take $A_j=\max(q,e^{2})$ and
\begin{eqnarray*}
  h_j&\le&\max \left(\log \left(\frac{n}{e}+\frac{2n+1/2}{A_j}\right), 1000, 782+142\log{\frac{p-1}{2}}\right).
\end{eqnarray*}
Then, \eqref{eq:BWineq1} is
\begin{eqnarray*}
  &\ge& \frac{2}{\pi^2}\exp \left(-1175\left(5.205+0.946\log{\frac{p-1}{2}}\right)(p-1)^4h\max(\log{q},2)\right),
\end{eqnarray*}
where $h=\max \left(\log \left(\frac{n}{e}+\frac{2n+1/2}{q}\right), 1000, 782+142\log{\frac{p-1}{2}}\right)$. If $p$ is fixed and $n$ is large enough with respect to it, this gives the expression in Corollary \ref{cor:lowerBounds}. This yields the result by Theorem \ref{thm:linIndepExp}. The argument is essentially the same to lower bound a single Kloosterman sum with Gouillon's result, with the analogue of \eqref{eq:BWineq1} having a leading factor of $2/\pi$, no product, and $A_0=\max(q/2,e)$.
\qed

\section{Angles of Gaussian primes}\label{sec:anglesGP}

\subsection{Definitions and cohomological interpretation}

\begin{definition}
  Let $q$ be an odd prime power and $k\ge 2$ be an integer. A \emph{super-even character $\Xi$ modulo $S^k$ over $\F_q$} is a character of
  \[\Sb_{k,q}^1\cong R_{k,q}/H_k, \qquad H_k:=\left(\F_q[S^2]/(S^k)\right)^\times\]
  (see \eqref{eq:Sk1}). The \emph{Swan conductor} of a non-trivial $\Xi$ is the maximal (odd) integer $d(\Xi)$ such that $\Xi$ is non-trivial on $\left(1+(S^{d(\Xi)})\right)/\big(S^k\big)\le R_{k,q}$. The character $\Xi$ is \emph{primitive} if $d(\Xi)=2\kappa-1$, with $\kappa:=\floor{k/2}$. The \emph{$L$-function} of a non-trivial $\Xi$ is
  \begin{equation}
    \label{eq:LXiT}
    L(\Xi,T)=\prod_{\substack{P\text{ prime}\\ \text{monic}\\P(0)\neq 0}}\left(1-\Xi(P)T^{\deg{P}}\right)^{-1}.
  \end{equation}
  
\end{definition}

\begin{theorem}[Katz]\label{thm:superEvenConstr}
  Let $\F_q$ be a finite field of odd characteristic $p$, $k\ge 2$ be an even integer,
  \[E=\Q \left(\zeta_{4p^r} : 1\le r\le 1+\frac{\log{k}}{\log{p}}\right)\subset\Q(\zeta_{p^\infty})\]
  with ring of integers $\Oc$, and let $\lambda\in\Lambda:=\Spec_p(\Oc)$.
  \begin{enumerate}
  \item There exists a unipotent group $\W_{k\codd}$ over $\F_p$ such that $\W_{k\codd}(\F_q)=\Sb_{k,q}^1$ (the group of super-even characters, by duality), as well as an open set $\Prim_{k\codd}\subset\W_{k\codd}$ such that $\Prim_{k\codd}(\F_q)$ is in bijection with primitive super-even characters modulo $S^k$ over $\F_q$.
  \item There exists a lisse sheaf $\Gc_{k,\lambda}$ on $\Prim_{k\codd}$ of free $\Oc_\lambda$-modules, of rank $r=2\kappa-2$, pure of weight $1$, such that for every $\Xi\in \Prim_{k\codd}(\F_q)$, we have
    \[\det\left(1-T\Frob_{q,\Xi}\mid \Gc_{k,\lambda}\right)=\frac{L(\Xi,T)}{1-T},\]
    which is a polynomial of degree $d(\Xi)=r+1$. In particular, the family $(\Gc_{k,\lambda})_{\lambda\in\Lambda}$ forms a compatible system.
  \item\label{item:superEvenTwist} The Tate twist $\Fc_{k,\lambda}=\Gc_{k,\lambda}(1/2)$ is a lisse sheaf of free $\Oc_\lambda$-modules on $\Prim_{k\codd}$, pure of weight zero, of rank $d(\Xi)-1$, with symplectic auto-duality.
  \end{enumerate}  
\end{theorem}
\begin{proof}
  These are the contents of \cite[Section 2]{Katz16} (see also the constructions in \cite[Sections 1-4]{Kat13}).
\end{proof}

In particular, the eigenvalues of $\Frob_{\F_q}$ acting on the stalks of $\Fc_{k,\lambda}$ at super-even primitive $\Xi$, which are free $\Oc_\lambda$-modules of rank $2\kappa-2$, yield the eigenvalues $e(\pm\theta_{\Xi,j})\in\C$ from \eqref{eq:eigenvaluesXi}, such that
\begin{eqnarray*}
  L(\Xi,T)&=&(1-T)\prod_{j=1}^{\kappa-1}\Big(1-\sqrt{q}e(\theta_{\Xi,j})T\Big)\Big(1-\sqrt{q}e(-\theta_{\Xi,j})T\Big)\\
          &=&(1-T)\det(1-\sqrt{q}T\Theta_\Xi), \quad \text{with } \Theta_\Xi\in\Sp_{d(\Xi)-1}(\C).
\end{eqnarray*}

\subsection{Existence of the limiting distribution}\label{subsec:limitingDiscExp}

We start with an explicit formula for $X_{k,N}(\bs u)$.
\begin{proposition}\label{prop:XkNun}
  For all $u\in\Sb_{k,q}^1$ and $n\le N$, we have
    \begin{eqnarray*}
      X_{k,N}(u)_n&=&-2\sum_{f=2}^\kappa\quad\sum_{j=1}^{f-1}\sum_{\substack{\Xi\in\hat\Sb_{k,q}^1\\d(\Xi)=2f-1}}\quad\overline\Xi(u)\cos(2\pi n\theta_{\Xi,j})\\
                  &&-\delta_{n \ \mathrm{even}}|\{b\in \Sb_{k,q}^1 : b^2=u\}|+O \left(\frac{q^{k/2}\tau(n)}{q^{n/6}n}+\frac{kq^k}{q^{n/4}}\right),
    \end{eqnarray*}
    with an absolute implied constant. Moreover, $|\{b\in \Sb_{k,q}^1 : b^2=u\}|\in\{0,1\}$ and in the expression above, $\overline\Xi(u)\cos(2\pi n\theta_{\Xi,j})$ may be replaced by $\Re(e(\theta_{\Xi,j})\Xi(u))$.
  \end{proposition}
  \begin{remark}
    Almost all (i.e. a density $1+O(1/q)$) super even $\Xi\in\hat\Sb^1_{k,q}$ have conductor $2\kappa-1$, but since we look at the $N\to\infty$ limit, we cannot restrict the sum in Proposition \ref{prop:XkNun} to those characters only as in \cite[Proof of Theorem 6.7]{RudWax17} (with a $q\to\infty$ limit).
  \end{remark}
  \begin{proof}
    By \cite[Lemma 6.4, Section 6.6]{RudWax17}, we have
    \begin{eqnarray*}
      X_{k,N}(u)_n&=&-\sum_{\Xi\neq 1} \overline\Xi(u)\tr\Theta_{\Xi}^n-\frac{R_{k,n}(u)q^\kappa}{q^{n/2}}-\frac{\delta_{u=1}q^\kappa}{q^{n/2}},
    \end{eqnarray*}
    where, by the prime polynomial theorem \cite[Theorem 2.2]{Ros02},
    \begin{eqnarray*}
      R_{k,n}(u)&:=&\sum_{\substack{f\in\F_q[S] \text{ monic}\\ \text{not prime}\\ \deg(f)=n}} \Lambda(f)\delta_{U(f)\in\Sec(u,k)}\\
                &=&\delta_{n\text{ even}}\frac{n}{2}\sum_{\substack{P\text{ monic}\\ \text{prime} \\ \deg(P)=n/2}} \delta_{U(P^{2})\in\Sec(u,k)}+ O \left(\frac{q^{n/3}\tau(n)}{n}\right).
    \end{eqnarray*}

    By the function field analogue of Dirichlet's theorem on primes in arithmetic progressions \cite[Theorem 4.8]{Ros02}, if $n$ is even,
    \begin{eqnarray*}
  &&-\frac{q^{\kappa}n}{2q^{n/2}}\sum_{\substack{P\text{ monic}\\ \text{prime}\\ \deg(P)=n/2}}\delta_{U(P^2)\in\Sec(u,k)}\\
  &=&-\frac{q^{\kappa}n}{2q^{n/2}}\sum_{\substack{a\in R_{k,q} \\ a^2\equiv u\pmod*{H_k}}} \sum_{\substack{P\text{ monic}\\\text{prime}\\ \deg(P)=n/2}} \delta_{P\equiv a\pmod*{S^k}}\\
    \end{eqnarray*}
    This is furthermore
    \begin{eqnarray*}
  &=&-\frac{q^{\kappa}n}{2q^{n/2}}\sum_{\substack{a\in R_{k,q} \\ a^2\equiv u\pmod*{H_k}}} \left(\frac{1}{|R_{k,q}|}\frac{q^{n/2}}{n/2}+O \left(\frac{q^{n/4}k}{n}\right)\right)\\
  &=&-|\{a\in R_{k,q} : a^2\equiv u\pmod*{H_k}\}|\left(\frac{1}{|H_k|}+O \left(\frac{q^{\kappa}k}{q^{n/4}}\right)\right)\\
  &=&-|\{b\in \Sb_{k,q}^1 : b^2=u\}|\left(1+O \left(\frac{k|R_{k,q}|}{q^{n/4}}\right)\right).
    \end{eqnarray*}
    Note that in odd characteristic, the cardinality $|\Sb_{k,q}^1|=q^\kappa$ is odd, so the function $(x\in\Sb_{k,q}^1)\mapsto x^2$ is injective, and $|\{b\in \Sb_{k,q}^1 : b^2=u\}|\in\{0,1\}$.
    
    Hence,
        \begin{eqnarray*}
          X_{k,N}(u)_n&=&-\sum_{\Xi\neq 1} \overline\Xi(u)\tr\Theta_{\Xi}^n-\frac{\delta_{u=1}q^\kappa}{q^{n/2}}+ O \left(\frac{q^\kappa\tau(n)}{q^{n/6}n}\right)\\
                      &&-\delta_{n\text{ even}}|\{b\in \Sb_{k,q}^1 : b^2=u\}|\left(1+O \left(\frac{kq^k}{q^{n/4}}\right)\right),
        \end{eqnarray*}
        which gives the result after splitting the sum over characters $\Xi$ depending on the conductors $d(\Xi)$, which are odd integers. The last assertion follows from the invariance of the sum under $\Xi\mapsto\overline\Xi$.
  \end{proof}

\begin{proof}[Proof of Theorem \ref{thm:limitDistrExp}] The existence of the limiting distribution goes almost exactly as in \cite[Lemma 3.1, Theorem 3.2]{Cha08} (based on \cite{RubSar94}). Let $\tilde X_{k,N}(\bs u)$ be the random variable on $[1,N]$ defined by the right-hand side of the expression in Proposition \ref{prop:XkNun}, but without the error term. Let moreover
  \begin{equation}
    \label{eq:VXij}
    V:=\{(\Xi,j) : \Xi\in\hat\Sb^1_{k,q}, \ \Xi\neq 1, \ 1\le j\le d'(\Xi)\}.
  \end{equation}
    There exists an explicit continuous function $g_{k,\bs u}:(\R/\Z)^V\to\R^R$ such that
    \[\tilde X_{k,N}(\bs u)=\Big(g_{k,\bs u}\left(n\theta_{\Xi,j} : (\Xi,j)\in V\right)\Big)_{n\le N.}\]
    Note that $g_{k,\bs u}$ is bounded is (when $k,q$ are fixed): each component is bounded by $2\kappa q^\kappa$.

    By the Kronecker--Weyl equidistribution theorem, $\left(n\theta_{\Xi,j} : (\Xi,j)\in V\right)_{n\le N}$ converges in law (as $N\to\infty$) to a random vector equidistributed in the closure $\overline\Gamma$ of the torus
    \begin{equation}
      \label{eq:Gamma}
      \Gamma=\left\{n\Big(\theta_{\Xi,j}\Big)_{(\Xi,j)\in V} : n\in\Z\right\}\subset (\R/\Z)^{V}.
    \end{equation}
    It then follows from Helly's selection theorem \cite[Theorems 25.9-10]{Bill86} that $X_{k,N}(\bs u)$ converges in law to a random vector $X_k(\bs u)$ which corresponds to a measure $\mu_{k,\bs u}$ satisfying
    \begin{equation}
      \label{eq:intRRGamma}
      \int_{\R^R} f(\bs x)d\mu_{k,\bs u}(\bs x)=\int_{\overline\Gamma}(f\circ g_{k,\bs u})(\bs x)d\bs x
    \end{equation}
    for every bounded continuous $f:\R^R\to\R$. The limiting measure $\mu_{k, \bs u}$ is compactly supported from the boundedness of $g_{k,\bs u}$ ($k, q$ fixed).

    In particular, there is convergence of the moments, which allows to compute the expected value by noting that
    \begin{eqnarray*}
      \left|\frac{1}{N}\sum_{f=2}^\kappa\quad\sum_{j=1}^{f-1}\sum_{\substack{\Xi\in\hat\Sb_{k,q}^1\\d(\Xi)=2f-1}}\overline\Xi(u)\sum_{n=1}^N\cos(2\pi n\theta_{\Xi,j})\right|&\ll&\frac{\kappa q^\kappa}{N}\xrightarrow{N\to\infty}0.
    \end{eqnarray*}
  \end{proof}

\subsection{Properties of the limiting distribution under (generic) linear independence}\label{subsec:propLimDistrGSHExp}

For the next properties, we continue to use the methods of Rubinstein--Sarnak \cite{RubSar94} and others, in particular by studying characteristic functions.

    \begin{lemma}[Fourier transform]\label{lemma:FTGP}
      For $u_1,\dots,u_R\in\Sb_{k,q}^1$ distinct, let $\mu_{k,\bs u}$ be the measure associated with the $R$-dimensional random vector $X_{k}(\bs u)$. Its Fourier transform
      \[\hat\mu_{k,\bs u}(\bs t):=\int_{\R^R} e^{-i\bs t\cdot \bs x}d\mu_{k,\bs u}(\bs x) \qquad (\bs t\in\R^R)\]
      is given by
      \[\exp \left(i\bs t\cdot \bs b_k(\bs u)\right)\int_{\overline\Gamma} \prod_{f=1}^{\kappa}\prod_{j=1}^{f-1}\prod_{\substack{\Xi\in\hat\Sb_{k,q}^1\\d(\Xi)=2f-1}} \exp \Big(2i \Re\left(e(x_j)\bs t\cdot\Xi(\bs u)\right)\Big) d\bs x,\]
    where $\Gamma$ is the torus \eqref{eq:Gamma} and $\bs b_k(\bs u):=(|\{b\in \Sb_{k,q}^1 : b^2=u_r\}|/2)_{1\le r\le R}$, $\Xi(\bs u):=(\Xi(u_r))_{1\le r\le R}$.
    If Hypothesis \ref{hyp:LI} holds, then
    \begin{equation}
      \label{eq:muhatGPLI}
      \hat\mu_{k,\bs u}(\bs t)=\exp \left(i\bs t\cdot \bs b_k(\bs u)\right)\prod_{f=2}^\kappa\prod_{j=1}^{f-1}\prod_{\substack{\Xi\in\hat\Sb_{k,q}^1\\d(\Xi)=2f-1}} J_0 \left(2|\bs t\cdot \Xi(\bs u)|\right),      
    \end{equation}
    where $J_0(z)=\frac{1}{\pi}\int_0^\pi \cos(z\sin{t})dt$ is the $0$th Bessel function of the first kind.
\end{lemma}
\begin{proof}
  The first statement is a direct consequence of Proposition \ref{prop:XkNun} and \eqref{eq:intRRGamma}. For \eqref{eq:muhatGPLI}, under Hypothesis \ref{hyp:LI} the torus $\overline\Gamma$ is maximal and the integral splits as a product of integrals of the form
  \[\int_{\R/\Z}\exp\Big(2i\Re \left(e(x_j)\bs t\cdot \Xi(\bs u)\right)\Big)dx_j=J_0 \left(2|\bs t\cdot \Xi(\bs u)|\right)\]
  by \cite[Lemma C.1]{MarNg17}.
\end{proof}
We now prove Theorem \ref{thm:distrPropGPLI} about properties of the limiting distribution under Hypothesis \ref{hyp:LI}.
\begin{proof}[Proof of Theorem \ref{thm:distrPropGPLI}]
  To show that $X_k(\bs u)$ is absolutely continuous, it is enough to show that $\int_{\R^R}|\hat\mu_{k,\bs u}(\bs t)|d\bs t<\infty$ (see \cite[Lemma A.8(b)]{MarNg17}). To do so, we partly follow the method of \cite[Section 4]{MarNg17}.  Since we assume Hypothesis \ref{hyp:LI}, we may use \eqref{eq:muhatGPLI} from Lemma \ref{lemma:FTGP}: we have
    \begin{eqnarray*}
      |\hat\mu_{k,\bs u}(\bs t)|&\le&\prod_{f=2}^\kappa\prod_{j=1}^{f-1}\prod_{\substack{\Xi\in\hat\Sb_{k,q}^1\\d(\Xi)=2f-1}} |\bs t\cdot \Xi(\bs u)|^{-1/2}\le\left[\prod_{\Xi\in S} |\bs t\cdot \Xi(\bs u)|^{2}\right]^{-\frac{\kappa-1}{4}}\\
      &\le&\left[\frac{1}{|S_1(\bs t)|}\sum_{\Xi\in S_1(\bs t)} |\bs t\cdot \Xi(\bs u)|^{2}\right]^{-\frac{\kappa-1}{4}}
    \end{eqnarray*}
    where
    \[S_1(\bs t):=\{\Xi\in\hat\Sb_{2\kappa,q}^1\text{ primitive }: |\bs t\cdot \Xi(\bs u)|>1\}\subset S:=\{\Xi\in\hat\Sb_{2\kappa,q}^1\text{ primitive}\},\]
    since $|J_0(z)|\le \min(1,\sqrt{2/(\pi|z|)})$ for all $z\in\R$ (see \cite[Lemma C.2]{MarNg17}). If $\bs t\in \bs T:=\{\bs t\in \R^R : |S_1(\bs t)|\ge 1\}$, we get
    \begin{eqnarray*}
      \frac{1}{|S_1(\bs t)|}\sum_{\Xi\in S_1(\bs t)} |t\cdot \Xi(\bs u)|^{2}&\ge&\frac{1}{|S|}\sum_{\Xi\in S}|t\cdot \Xi(\bs u)|^{2}\\
                                                              &=&\sum_{r,r'=1}^R t_r\overline{t_{r'}}\frac{1}{|S|}\sum_{\Xi\in S}\Xi(u_r)\overline\Xi(u_{r'}).
    \end{eqnarray*}
    By the orthogonality relations and Möbius inversion,
    \begin{eqnarray*}
      \frac{1}{|S|}\sum_{\Xi\in S}\Xi(u_r)\overline\Xi(u_{r'})&=&\frac{1}{|S|}\sum_{f=2}^\kappa\mu(S^{2(\kappa-f)})\sum_{\Xi\in\hat\Sb_{2f,q}^1}\Xi(u_r)\overline\Xi(u_{r'})\\
                                                              &=&\frac{q^{\kappa}\delta_{u_r=u_{r'}}}{|S|}=\frac{\delta_{u_r=u_{r'}}}{1-1/q}.
    \end{eqnarray*}
    Since the $u_i$ are distinct, it follows that
    \[\frac{1}{|S_1(\bs t)|}\sum_{\Xi\in S_1(\bs t)} |\bs t\cdot \Xi(\bs u)|^{2}\ge||\bs t||^2\qquad\text{if }|S_1(\bs t)|\ge 1.\]
    Therefore, if $\bs t\in\bs T$, then $|\hat\mu_{k,\bs u}(\bs t)|\le ||\bs t||^{-\frac{\kappa-1}{2}}$.
    On the other hand, if $\bs t\not\in \bs T$, the same argument shows that
    \[1\ge \frac{1}{|S|}\sum_{\Xi\in S}|\bs t\cdot \Xi(\bs u)|\ge ||\bs t||^2,\]
    i.e. $\R^R\backslash \bs T$ is bounded. It also contains a neighborhood of $\bs 0$ since it contains the finite intersection $\bigcap_{\Xi\in S} \{\bs t\in\R^R: |\bs t\cdot \Xi(\bs u)|<1\}$ of open sets containing $\bs 0$.

    Thus, there exists $\varepsilon>0$ such that
    \begin{eqnarray*}
      \int_{\R^R}|\hat\mu_{k, \bs u}(\bs t)|d\bs t&\ll&\int_{||\bs t||\le 1} |\hat\mu_{k,\bs u}(\bs t)| d\bs t+\int_{\R^R\backslash B_\varepsilon(\bs 0)}||\bs t||^{-\frac{\kappa-1}{2}}d\bs t,
    \end{eqnarray*}
    and the second integral converges when $\kappa-1>2R$ (see \cite[p. 22]{MarNg17}). This concludes the proof of \ref{item:thm:distrPropGPLIABsCont}.
    
    Concerning \ref{item:thm:distrPropGPLISymmetry}, the symmetry/exchangeability follow from the expression \eqref{eq:muhatGPLI} for $\hat\mu_{k,\bs u}$.

        The last statements of the theorem follow from the previous ones: since $\mu_{k,\bs u}$ is absolutely continuous, $A=\{\bs x\in\R^R:x_1<\dots<x_R\}$ is a continuity set, so that by the portmanteau theorem,
    \[\lim_{N\to\infty}\P\Big(X_{k,N}(u_1)<\dots<X_{k,N}(u_R)\Big)=\mu_{k,\bs u}(A).\]
  \end{proof}

  Finally, we prove Corollary \ref{cor:distrPropGPweak} (unconditional properties of the limiting distribution) assuming Theorem \ref{thm:genericLIGP} on generic linear independence.
\begin{proof}[Proof of Corollary \ref{cor:distrPropGPweak}] ~
  
  \begin{enumerate}
  \item It suffices to show it when $R=1$, i.e. that the random variable $X_k(u)$ is continuous for every $u\in \Sb_{k,q}^1$. We follow the argument in \cite[Proof of Theorem 2.2]{Dev18} (see also \cite[Proposition 2.1]{DevMeng18}). By Wiener's lemma, it suffices to show that
  \begin{equation}
    \label{eq:Wiener}
    \lim_{S\to\infty} \frac{1}{S}\int_{-S}^S |\hat\mu_{k,u}(t)|^2dt=0.
  \end{equation}
  By Lemma \ref{lemma:FTGP}, $|\hat\mu_{k,u}(t)|\le\left|\int_{\overline\Gamma} \exp \left(it \phi(\bs x)\right)d\bs x\right|$, where
  \begin{eqnarray*}
    \phi(\bs x)&:=&2\sum_{f=1}^{\kappa}\sum_{j=1}^{f-1}\sum_{\substack{\Xi\in\hat\Sb_{k,q}^1\\d(\Xi)=2f-1}}\cos(2\pi x_j) \Xi(u).
  \end{eqnarray*}
  By Theorem \ref{thm:genericLIGP}, there exists $\Xi\in\hat\Sb_{k.q}^{1}$ and $1\le j\le d'(\Xi)$ such that $\theta_{\Xi,j}\not\in\Q$. It follows that the function $\phi: \overline\Gamma\to\R$ is analytic and non-constant, since $\Xi(u)\neq 0$ (being a root of unity). Thus, the scaling principle \cite[VIII.2, Proposition 5]{Ste93} shows that $|\hat\mu_{k,u}(t)|\ll |t|^{-\alpha}$ for some constant $\alpha>0$, where $\alpha$ and the implied constant can depend on all parameters but $t$. Thus, \eqref{eq:Wiener} holds, using the trivial bound $|\hat\mu_{k,u}(t)|\le 1$ around $0$.
\item This is a consequence of the proof of Theorem \ref{thm:limitDistrExp}: $\overline\Gamma$ is a subtorus of $(\R/\Z)^{V}$, with $V$ as in \eqref{eq:VXij}), and if the set of the $\theta_{\Xi,j}$ ($(\Xi,j)\in V$) contains at least $t$ linearly independent elements, then $\dim\overline\Gamma\ge t$. By Theorem \ref{thm:genericLIGP}, the latter holds whenever $t=o(\log|\Sb_{k,q}^1|)$.
  \end{enumerate}
\end{proof}

\section{Prime polynomials in short intervals}\label{sec:variancePPSI}
\subsection{Definitions and cohomological interpretation}
\begin{definition}
  Let $Q\in\F_q[T]$ be non-constant.
  \begin{itemize}
  \item A Dirichlet character $\chi$ modulo $Q$ is a character of $(\F_q[T]/(Q))^\times$.
  \item The character $\chi$ is \textit{even} if it is trivial on $\F_q^\times$.
  \item It is \textit{primitive} if it is not induced from a character modulo a proper divisor $Q'\mid Q$ through the natural map $(\F_q[T]/(Q))^\times\to(\F_q[T]/(Q'))^\times$. The \textit{conductor} of $\chi$ is the monic divisor $Q'\mid Q$ of smallest degree such that $\chi$ is primitive modulo $Q'$.
  \item As usual, we may extend $\chi$ as $\chi:\F_q[T]\to\C$ by defining $\chi(f)=\chi(f\pmod*{Q})$ if $(f,Q)=1$, $\chi(f)=0$ otherwise.
\item The number of Dirichlet characters modulo $Q$ is denoted by $\varphi(Q)$. The number of even (resp. primitive, even primitive) such characters is $\varphi^\ev(Q)=\varphi(Q)/(q-1)$ (resp. $\varphi_\prim(Q)$, $\varphi_\prim^\ev(Q)$).
\item The \emph{$L$-function} of $\chi$ is
  \[L(\chi,T)=\prod_{\substack{P\text{ prime}\\ \text{monic}\\P\nmid Q}}\left(1-\chi(P)T^{\deg{P}}\right)^{-1}.\]
  \end{itemize}    
\end{definition}

We recall that if $\deg(Q)\ge 2$ and $\chi\neq 1$, then $L(\chi,T)$ is a polynomial (rather than a formal power series) of degree $\deg(Q)-1$ (see \cite[Proposition 4.3 and p. 130]{Ros02}).

  If $\chi$ is even, then $L(\chi,T)$ has a ``trivial'' zero at $T=1$. As in \cite[(3.34)]{KeatRud14}, we define $\lambda_\chi=\delta_{\chi\text{ even}}$, which allows to factor
  \[L(\chi,T)=(1-\lambda_\chi T)L^*(\chi,T), \quad L^*(\chi,T)\in\F_q[T].\]
  If $\chi$ is primitive, Weil's work on the Riemann hypothesis over finite fields (see \cite[Chapters 4, 5]{Ros02}) shows that
  \begin{equation}
    \label{eq:L*theta}
    L^*(\chi,T)=\det(1-\sqrt{q}T\Theta_\chi), \qquad \Theta_\chi\in U_{\deg(Q)-1-\lambda\chi}(\C),
  \end{equation}
  and we let
  \[e(\theta_{\chi,j}),\quad (1\le j\le \deg(Q)-1-\lambda_\chi), \quad \theta_{\chi,j}\in[0,1],\]
  be the eigenvalues of $\Theta_\chi^{-1}$. This is also reflected in the following result:
\begin{theorem}[Katz]\label{thm:chiConstr}
  Let $\F_q$ be a finite field of odd characteristic $p$, $m\ge 2$ be an integer,
  \[E=\Q \left(\zeta_{m-2},\zeta_{4p^r} : 1\le r\le 1+\frac{\log{m}}{\log{p}}\right)\subset\Q(\zeta_{p^\infty},\zeta_n)\]
  with ring of integers $\Oc$, and let $\lambda\in\Lambda:=\Spec_p(\Oc)$.
  \begin{enumerate}
  \item There exists a unipotent group $\W_m$ over $\F_p$ such that $\W_m(\F_q)$ is the group of even characters modulo $T^m\in\F_q[T]$, as well as an open set $\Prim_{m}\subset\W_m$ such that $\Prim_{m}(\F_q)$ is the set of primitive even characters modulo $T^m$.
  \item There exists a lisse sheaf $\Gc_{m,\lambda}$ on $\Prim_{m}$ of free $\Oc_\lambda$-modules, of rank $m-2$, pure of weight $1$, such that for every $\chi\in\Prim_m(\F_q)$,
    \[\det\left(1-T\Frob_{q,\chi}\mid \Gc_{m,\lambda}\right)=L^*(\chi,T),\]
    which is a polynomial of degree $m-2$. In particular, the family $(\Gc_{m,\lambda})_{\lambda\in\Lambda}$ forms a compatible system.
  \item\label{item:chiTwist} The Tate twist $\Fc_{m,\lambda}=\Gc_{m,\lambda}(1/2)$ is a lisse sheaf of free $\Oc_\lambda$-modules on $\Prim_{m}$, pure of weight zero, of rank $m-2$.
  \end{enumerate}
\end{theorem}
In other words, the eigenvalues of $\sqrt{q}\Theta_\chi$ (the zeros of $L^*(\chi,T)$) are the eigenvalues of $\Frob_{\F_q}$ acting on the stalk of $\Gc_{m,\lambda}$ at $\chi$.
\begin{proof}
  This is essentially the contents of \cite[Sections 1-4]{Kat13}. The addition of $\zeta_{m-2}$ is not necessary at this point, but will be useful in Theorem \ref{thm:monEvenDir}.
\end{proof}
\subsection{Existence of the limiting distribution}

We start with an explicit formula for $X_{m,N}(\bs B)$, and proceed as in Section \ref{subsec:limitingDiscExp}.

\begin{proposition}\label{prop:XkNunVar}
  Under the notations of Section \ref{subsec:PPSI}, we have, for $B\in\F_q[T]$ monic of degree $m-1$,
  \begin{eqnarray*}
    X_{m,N}(B)_n&=&-\sum_{f=3}^m\sum_{\substack{\chi\pmod*{T^m}\\ \mathrm{even}\\ \cond(\chi)=T^f}}\sum_{j=1}^{f-2}\quad \overline\chi(B^*)e(\theta_{\chi,j})+\frac{1}{q^{n/2}},
  \end{eqnarray*}
  where $B^*\in\F_q[T]$ is the reflected polynomial defined by $B^*(T)=T^{\deg B}B(1/T)$.
\end{proposition}
\begin{proof}
  By \cite[(4.22)]{KeatRud14},
  \begin{eqnarray*}
    X_{m,N}(B)_n&=&\frac{1}{q^{n/2}}\sum_{\substack{\chi\pmod*{T^m} \\ \text{even}}}\overline\chi(B^*)\psi(n,\chi),\\
    \psi(n,\chi)&:=&\sum_{\substack{f\in\F_q[T]\\ \deg(f)=n}} \Lambda(f)\chi(f)=-q^{n/2}\tr(\Theta_\chi^n)-1,
  \end{eqnarray*}
  where the last equality is the explicit formula for $\psi$ (see \cite[(3.38)]{KeatRud14}), obtained by taking the logarithmic derivative on both sides of \eqref{eq:L*theta}. Thus,
  \begin{eqnarray*}
    X_{m,N}(B)_n&=&\frac{1}{q^{n/2}}\sum_{\substack{\chi\pmod*{T^m} \\ \text{even}}}\overline\chi(B^*)\tr(\Theta_\chi^n)-\frac{1}{q^{n/2}}\sum_{\substack{\chi\pmod*{T^m} \\ \text{even}}}\overline\chi(B^*).
  \end{eqnarray*}
  The result follows after splitting the first sum according to the conductor of $\chi$ and applying the orthogonality relations in $(\F_q[T]/(T^m))^\times/\F_q^\times$ to the second one.
\end{proof}

Then, the Proof of Theorem \ref{thm:limitDistrExpVar} is exactly like the proof of Theorem \ref{thm:limitDistrExp} (see Section \ref{subsec:limitingDiscExp}). As in Proposition \ref{prop:XkNun}, one may replace the $e(\theta_{\chi,j})$ in Proposition \ref{prop:XkNunVar} by $\cos(2\pi\theta_{\chi,j})$ since $X_{m,N}(B)_n\in \R$.

\subsection{Properties of the limiting distribution under (generic) linear independence}
Again, the proofs of Theorem \ref{thm:distrPropGPLIchi} and Corollary \ref{cor:distrPropGPweakchi} are exactly like the proofs of Theorem \ref{thm:distrPropGPLI} and Corollary \ref{cor:distrPropGPweak} respectively, in Section \ref{subsec:propLimDistrGSHExp}.

\section{An extension of the large sieve for Frobenius}\label{sec:largeSieve}

In the next two sections, we set up the tools to prove the main Theorems \ref{thm:linIndepExp}, \ref{thm:genericLIGP} and \ref{thm:genericLIGPchi} on generic linear independence. As outlined in Section \ref{subsec:methods}, the strategy follows that of previous works and is the following:
\begin{enumerate}
\item Obtain information about integral monodromy groups of reductions of sheaves of $\Oc_\lambda$-modules from Theorem \ref{thm:Klconstr} and \ref{thm:superEvenConstr}, for a set of ideals/valuations $\lambda\in\Spec_{1,p}(\Oc)$ of positive density.
\item\label{item:LSConsAlgRel} Use a variant of the large sieve for Frobenius to show that for all such $\lambda$, the (splitting) fields generated by the roots ($\alpha_{i,f,p}(x)$, $e(\theta_{\Xi,j})$ or $e(\theta_{\chi,j})$) are maximal for almost all tuples of arguments $x$ (resp. $\Xi,\chi$) for exponential sums (resp. (super-)even characters).
\item Apply Girstmair's work to show that \ref{item:LSConsAlgRel} implies the desired linear independence.
\end{enumerate}
The first two points and the variant of the large sieve for Frobenius are implemented in this section, and the third point in Section \ref{sec:genericMax}.
\begin{remark}
  Note that \cite{Kow08,ChaFioJou16b} dealt with symplectic and orthogonal monodromy types. Here, we need to consider special linear and symplectic ones, which will correspond to splitting fields with Galois groups $\Sf_n$ (the full symmetric group), or $W_{2n}\le \Sf_{2n}$, the subgroup with order $2^nn!$ of permutations of $n$ pairs (the Coxeter group $B_n$).
\end{remark}
\begin{remark}
  We consider ideals of degree $1$ so that $\F_\lambda=\F_\ell$ and considerations on the sheaves mod $\lambda$ can be reduced as much as possible to existing arguments, for the large sieve or computations of integral monodromy groups. This is actually not a restriction because $\Spec_{1,p}(\Oc)$ has natural density $1$ in $\Spec(\Oc)$ (\cite[Corollary 2, p. 345, Proposition 7.17]{Nark04})
\end{remark}
\begin{remark}
    Since we considered Tate-twisted/normalized sheaves of $\Oc_\lambda$-modules from the beginning (which also forces the determinant to be trivial and the arithmetic/geometric monodromy groups to coincide, for exponential sums and super-even characters), we will not encounter the difficulty observed in \cite{Kow08,ChaFioJou16b} that the normalized characteristic polynomials may be defined over a quadratic extension of the base field, with the possibility of a different Galois group. This was overcome in ibid. by looking at squares of the roots, and showing that their Galois group was still maximal from a study of additive relations, in addition to the multiplicative ones.
  \end{remark}
\subsection{Integral monodromy groups}\label{subsec:intMonGroups}

The lisse sheaves $\Fc_\lambda$ of free modules on a variety $X$ given by Theorems \ref{thm:Klconstr}, \ref{thm:superEvenConstr} and \ref{thm:chiConstr} correspond to continuous representations $\rho_\lambda: \pi_1(X,\overline\eta)\to\GL_r(\Oc_\lambda),$ for $\overline\eta$ a geometric generic point, such that for every $x\in X(\F_q)$, if $\Frob_{x,q}\in \pi_1(X,\overline\eta)^\sharp$ is the geometric Frobenius conjugacy class at $x$, then $\rho_\lambda \left(\Frob_{x,q}\right)\in\GL_r(\Oc_\lambda)^\sharp$ gives the action of $\Frob_{q}$ on $(\Fc_\lambda)_{x}$.
\begin{definition}[Monodromy groups]
The \emph{geometric and arithmetic monodromy groups} of $\rho_\lambda$ are respectively
\[G_\lambda^\geom:=\overline{\rho_\lambda\Big(\pi_1^{\text{geom}}(X,\overline\eta)\Big)}^\Zar\le G_\lambda:= \overline{\rho_\lambda\Big(\pi_1(X,\overline\eta)\Big)}^\Zar\le\GL_r(\overline{E_\lambda}),\]
where $\overline{\,\cdot\,}^\Zar$ denotes Zariski closure in $\GL_r(\overline{E_\lambda})$. By reducing modulo $\lambda$, we also obtain representations $\tilde \rho_\lambda: \pi_1(X,\overline\eta)\to\GL_r(\F_\lambda)$, and we define the \emph{geometric and arithmetic integral monodromy groups} of $\rho_\lambda$ as the monodromy groups
\[\tilde G_\lambda^\geom:=\tilde\rho_\lambda\Big(\pi_1^{\text{geom}}(X,\overline\eta)\Big)\le \tilde G_\lambda:= \tilde\rho_\lambda\Big(\pi_1(X,\overline\eta)\Big)\le\GL_r(\F_\lambda)\]
of $\tilde\rho_\lambda$. If the adjective ``projective'' is added to those groups, one refers to their image with respect to the projections $\GL_r\to\PGL_r$ (over $\overline{E_\lambda}$ or $\F_\lambda$ respectively).
\end{definition}

\subsubsection{From monodromy to integral monodromy}

The determination of integral monodromy groups may be more challenging that their counterparts over $\overline{E_\lambda}$, since they have less structure (under purity assumption, the connected component at the identity of $G_\lambda^\geom$ is a semisimple algebraic group).

Fortunately, as explained by \cite[Section 7]{Katz12}, one may use deep results of Larsen--Pink \cite{LarsPink92,Lars95} to conclude (roughly) that if the monodromy over $\overline{E_\lambda}$ is as large as possible, then the same holds for a density $1$ of the integral monodromy groups.

Katz's argument is given for sheaves of $\Z_\ell$-modules, but carries over more generally to sheaves of $\Oc_\lambda$-modules: we spelled out the details in \cite[Section 5.2]{PG18}, and the conclusion reads as:

\begin{theorem}\label{thm:KatzLP}
  Let $X$ be a smooth affine geometrically connected variety over $\F_p$, let $E\subset\C$ be a Galois number field with ring of integers $\Oc$, and let $\Lambda$ be a set of valuations on $\Oc$ of natural density 1. Let $(\Fc_\lambda)_{\lambda\in\Lambda}$ be a compatible system with $\Fc_\lambda$ a lisse sheaf of free $\Oc_\lambda$-modules on $X$. We assume that
  \begin{quotation}
    there exists $G\in\{\SL_n,\Sp_{2n}\}$ such that for every $\lambda\in\Lambda$, the arithmetic monodromy group of $\Fc_\lambda$ is conjugate to $G(\overline{E_\lambda})$.
  \end{quotation}
  Then there exists a subset $\Lambda_p\subset\Lambda\cap\Spec_{1,p}(\Oc)$ of natural density 1, depending on $p$ and on the family, such that  $\Fc_\lambda$ has geometric \emph{and} arithmetic integral monodromy groups conjugate to $G(\F_\lambda)$ for all $\lambda\in\Lambda_p$.
\end{theorem}
\begin{remark}[Implied constants]
  The dependency of the sets of valuations on some of the variables $p,k,m$ in Theorems \ref{thm:monBirch}, \ref{thm:monSuperEven} and \ref{thm:monEvenDir} below will give dependencies on those of the implied constants in the final results.
\end{remark}
\begin{remark}[Strong approximation]
  Another method to get information on integral monodromy groups from the transcendental ones is through strong approximation results for arithmetic groups, as explained in \cite[Section 9]{Katz12} (see also \cite[Section 5]{JKZ13}); this is for example used in \cite{ChaFioJou16b}. In those cases, \cite{Pink00} (a generalization of \cite{MVW84,Weis84}) allows to show that the integral monodromy is large for all but finitely many primes. Moreover, by also using results of \cite{LarsPink92}, it avoids the classification of finite simple groups, unlike \cite{MVW84,Weis84}.

  However, this requires that the sheaves $\Fc_\lambda$ on $X$ may be formed over the analytification $X^\an$: a sheaf $\Fc^\an$ of finitely generated $\Oc$-modules is constructed on $X^\an$, whose extension of scalars to $\Oc_\lambda$ corresponds to the analytification of $\Fc_\lambda$, and strong approximation can then be applied to the monodromy of $\Fc^\an$ in $G(\Oc)$ to yield the result. This can be done in the case of families of $L$-functions considered in \cite{Katz12,ChaFioJou16b}, but a priori not for the sheaves from Theorems \ref{thm:Klconstr} and \ref{thm:superEvenConstr} (one may think about Artin--Schreier sheaves, i.e. Kloosterman sheaves of rank 1, as a first example)
\end{remark}

\subsubsection{Kloosterman and Birch sheaves}
Combining Theorem \ref{thm:KatzLP} with the determination of monodromy groups over $\overline E_\lambda$ by Katz, we obtain the following:

\begin{theorem}[Kloosterman sheaves]\label{thm:monKl}
  In the setting of Theorem \ref{thm:Klconstr}, there exists a subset $\Lambda_{r,p}$ of $\Spec_{1,p}(\Oc)$, of natural density $1$, such that for every $\lambda\in\Lambda_{r,p}$, the arithmetic and geometric integral monodromy groups of $\Klc_{r,\lambda}$ are equal and conjugate to $\SL_r(\F_\lambda)$ if $r$ is odd, $\Sp_r(\F_\lambda)$ if $r$ is even.
\end{theorem}
\begin{proof}
  This follows from Theorem \ref{thm:KatzLP} and the determination of monodromy groups over $\overline E_\lambda$ contained in \cite[Chapter 11]{KatzGKM}.
\end{proof}
\begin{remark}\label{rem:monKl}
  By work of Hall \cite{Hall08} or J-K. Yu (unpublished) when $r=2$, and the author \cite{PGIntMonKS16} for any $r\ge 2$, one may actually take
  \begin{equation}
    \label{eq:LambdaKlc}
    \Lambda_{r,p}=\{\lambda\in\Spec_{1,p}(\Oc)\text{ above }\ell : \ell\gg_r 1\}.
  \end{equation}
  In particular, the densities of elements $\Lambda_{r,p}$ with bounded norm are bounded from below independently of $p$.
\end{remark}

\begin{theorem}[Birch sheaves]\label{thm:monBirch}  
  In the setting of Theorem \ref{thm:Klconstr}\ref{item:Biconstr}, there exists a subset $\Lambda_{p}$ of $\Spec_{1,p}(\Oc)$, of natural density $1$, such that for every $\lambda\in\Lambda_p$, the arithmetic and geometric integral monodromy groups of $\Bic_\lambda$ are equal to $\SL_2(\F_\lambda)$.
\end{theorem}
\begin{proof}
  This follows from Theorem \ref{thm:KatzLP} and the determination of monodromy groups over $E_\lambda$ in \cite[7.12]{KatzESDE}.
\end{proof}
\subsubsection{Primitive super-even characters}
\begin{theorem}\label{thm:monSuperEven}
  In the setting of Theorem \ref{thm:superEvenConstr} \ref{item:superEvenTwist}, assuming that $k\ge 4$, there exists a subset $\Lambda_{k,p}\subset\Spec_{1,p}(\Oc)$ of natural density $1$ such that for every $\lambda\in\Lambda_{k,p}$, the arithmetic and geometric integral monodromy groups of $\Gc_{k,\lambda}$ are equal and conjugate to $\Sp_{2\kappa-2}(\F_\lambda)$.
\end{theorem}
\begin{proof}
  This follows from Theorem \ref{thm:KatzLP} and the determination of monodromy groups over $E_\lambda$ in \cite[Theorem 2.5]{Katz16} (using results from \cite[3.10]{KatzMMP}).
\end{proof}
\subsubsection{Primitive even characters mod $T^m$}
 
  \begin{theorem}\label{thm:monEvenDir}
    In the setting of Theorem \ref{thm:chiConstr} \ref{item:chiTwist}, assuming that $m\ge 5$ is odd, there exists a subset $\Lambda_{m,p}\subset\Spec_{1,p}(\Oc)$ of natural density 1 such that for every $\lambda\in\Lambda_{m,p}$, the projective arithmetic and geometric integral monodromy groups of $\Gc_{m,\lambda}$ are conjugate to $\PSL_{m-2}(\F_\lambda)$. 
\end{theorem}
\begin{proof}
  By, \cite[Theorem 5.1]{Kat13},
  \[\SL_{m-2}(\C)\le G_\geom(\Gc_{m,\lambda})\le G_\arith(\Gc_{m,\lambda})\le\GL_{m-2}(\C),\]
  whence $PG_\geom(\Gc_{m,\lambda})=PG_\arith(\Gc_{m,\lambda})=\PGL_{m-2}(\C)$.

  However, projective representations are not directly handled in Theorem \ref{thm:KatzLP}. Instead, we note that if $\lambda\in\Spec_{1,p}(\Oc)$ is above $\ell\nmid m-2$, then $\ell\equiv 1\pmod{m-2}$ (by the characterization of ideals of degree $1$ in cyclotomic extensions), so Hensel's lemma implies that every element of $\Oc_\lambda$ has an $(m-2)$th root, whence $\PGL_{m-2}(\Oc_\lambda)\cong\SL_{m-2}(\Oc_\lambda)$.

  If $\Gc_{m,\lambda}$ corresponds to a representation $\rho_\lambda: \pi_1(X,\overline\eta)\to\GL_{m-2}(\Oc_\lambda)$ and $\pi: \GL_{m-2}\to\PGL_{m-2}$ is the projection, we get in this case a continuous representation $\pi\circ\rho_\lambda: \pi_1(X,\overline\eta)\to\SL_{m-2}(\Oc_\lambda)$ with transcendental arithmetic and geometric monodromy groups isomorphic to $\SL_{m-2}(\C)$. We may then apply Theorem \ref{thm:KatzLP} to get that the arithmetic and geometric integral monodromy of $\pi\circ\rho_\lambda$ are $\SL_{m-2}(\F_\lambda)\cong\PGL_{m-2}(\F_\lambda)$ for a subset of density $1$ of $\Spec_{1,p}(\Oc)$. Since $\im(\pi\circ\rho_\lambda\pmod*{\lambda})=\pi(\im\rho_\lambda \pmod*{\lambda})$, this proves the assertion on the projective monodromy groups of $(\Gc_{m,\lambda})_{\lambda}$.  
\end{proof}
\subsection{Large sieve for Frobenius, with wild ramification}
Next, we need a version of the large sieve for Frobenius, originally developed in \cite{KowLS06} (see also \cite{KowLargeSieve08,Kow08}).

In these works as well as in \cite{ChaFioJou16b}, the sieve applies to sheaves of $\F_\ell$-modules on a variety $X$ over $\F_p$, that either:
\begin{enumerate}
\item\label{item:compSystems} are compatible systems, with $X$ a curve;
\item\label{item:tameRam} are tamely ramified;
\item\label{item:primeTop} have monodromy group of cardinality prime to $p$, a stronger condition than the previous one.
\end{enumerate}
For Kloosterman and Birch sums, \ref{item:compSystems} applies. However, for super-even characters, the variety is not a curve, and the sheaves are a priori not tamely ramified, which rules out \ref{item:tameRam}. Concerning \ref{item:primeTop}, note that for $E=\Q(\zeta_{p^N})$ and $\lambda\in\Spec(\Oc)$, the prime $p$ always divides $|\SL_r(\F_\lambda)|$ and $|\Sp_r(\F_\lambda)|$ (if $r$ is even).

\subsubsection{Extension of the large sieve for Frobenius}

Instead, we give an extension of \cite[Theorem 3.1]{KowLS06}/\cite[Theorems 4.1, 4.3]{Kow08} that works in this case and answers the question in \cite[Remark 4.8]{KowLS06}. To bound the sums of Betti numbers that appear, we give two arguments:
\begin{enumerate}
\item One, Theorem \ref{thm:largeSieve}\ref{item:modular}, involving sums of Betti numbers associated to tensor powers of the sheaves, inspired by \cite[Section 4]{KowLS06}, \cite[Theorem 9.2.6]{KatzSarnak91}, \cite[Lemma 5.2]{Katz16}, and an effective/modular version of a theorem of Burnside on irreducible representations contained in tensor powers of faithful representations.
\item Another, Theorem \ref{thm:largeSieve}\ref{item:compSystem}, provided by Will Sawin, reducing to the tame case (where a result of Deligne \cite{Ill81} on the Euler characteristic of tamely ramified sheaves can be applied) by exploiting the presence of a compatible system. This gives a much stronger bound, but with less explicit constants.
\end{enumerate}

\begin{definition}\label{def:Betti}
  Let $X$ be a smooth affine geometrically connected algebraic variety over $\F_p$, $E$ be a number field with ring of integers $\Oc$, let $\lambda,\lambda'\in\Spec_{1,p}(\Oc)$, and let $\Fc$ be a lisse sheaf of $R$-modules on $X$, where $R=\overline\Q_\ell$, $\Oc_\lambda$, $\Oc_\lambda\otimes\Oc_{\lambda'}$, $\F_\lambda$, or $\F_\lambda\otimes\F_{\lambda'}$. We define the sum of Betti numbers
  \[\sigma_c(X,\Fc)=\sum_{i=0}^{2\dim X} \rank H_c^i(X,\Fc),\]
  where the rank of an $R$-module is defined as its dimension over the total ring of fractions of $R$ (recall that these cohomology groups are finitely generated by \cite[Exposé 1, Théorème 4.6.2]{DelEC}).

  If $X$ is a curve and $R=\Oc_\lambda$, we moreover define
  \[\cond(\Fc_\lambda)=1-\chi_c(X,\Q_\ell)+2\sum_{x}\Swan_x(\Fc_\lambda)\]
  to be the quantity in \cite[(4.1)]{KowLS06} (see also \cite[Chapters 1--2]{KatzGKM}), where the sum is over ``points at infinity'' of $X$.
\end{definition}

\begin{theorem}\label{thm:largeSieve}
  Let $X$ be a smooth affine geometrically connected algebraic variety of dimension $d$ over $\F_p$. For $E$ a number field with ring of integers $\Oc$, let $\Lambda\subset\Spec_{1,p}(\Oc)$ with lower density
  \begin{equation*}
    \label{eq:deltaLambda}
    \delta_\Lambda:=\liminf_{L\to\infty} \frac{|\{\lambda\in\Lambda : N(\lambda)\le L\}|}{L/\log{L}}>0.
  \end{equation*}
  For every $\lambda\in\Lambda$, let $\Fc_\lambda$ be a rank $r$ lisse sheaf of $\F_\lambda$-modules on $X$, corresponding to a representation
  \begin{equation}
    \label{eq:rholambda}
    \rho_\lambda: \pi_1(X,\overline\eta)\to\GL_{r}(\F_\lambda),
  \end{equation}
  for $\overline\eta$ a geometric generic point. We assume that there exists $G\in\{\SL_r, \Sp_r\}$ such that either:
  \begin{enumerate}[(i)]
  \item\label{item:notPG} the arithmetic and geometric monodromy groups of $\rho_\lambda$ are equal and conjugate to $G(\F_\lambda)$ for all $\lambda\in\Lambda$, or;
  \item\label{item:PG} the \emph{projective} arithmetic and geometric monodromy group of $\rho_\lambda$ are equal and conjugate to $\PGL_r(\F_\lambda)$ for all $\lambda\in\Lambda$, and $\zeta_r\in E$, so that\footnote{See the proof of Theorem \ref{thm:monEvenDir}, recalling that $\lambda$ has degree $1$.} $\PGL_r(\F_\lambda)=\SL_r(\F_\lambda)=G(\F_\lambda)$.
  \end{enumerate}

  Let $t\ge 1$ be an integer. For every $\lambda\in\Lambda$, let $\bs\Omega_\lambda\subset G(\F_\lambda)^t$ be a conjugacy-invariant subset, such that
  \[\delta_{\bs\Omega}:=\sup_{\lambda\in\Lambda} \frac{|\bs\Omega_\lambda|}{|G(\F_\lambda)|^t}<1.\]
  Then, for any field $\F_q$ of characteristic $p$ and any $L\ge 1$,
  \begin{eqnarray*}
    P \Big(q,(\Fc_\lambda,\bs\Omega_\lambda)_{\lambda\in\Lambda}\Big)&:=&\frac{|\{\bs x\in X(\F_q)^t : (\rho_\lambda(\Frob_{x_i,q}))_i\in\bs\Omega_\lambda\text{ for all }\lambda\in\Lambda\}|}{|X(\F_q)|^t}\\
                                                                  &\ll&\frac{1}{(1-\delta_{\bs \Omega})\delta_\Lambda }\frac{\log{L}}{L}\left(1+\frac{tC \left(L,(\Fc_\lambda)_{\lambda\in\Lambda}\right)^t}{q^{1/2}}\right),
  \end{eqnarray*}
  where
  \begin{enumerate}[(a)]
  \item\label{item:curves} If $d=1$, $C\left(L,(\Fc_\lambda)_{\lambda\in\Lambda}\right) \ll r^{\delta_{G=\SL_r}}L^{\dim G+\frac{\rank G}{2}} \max\limits_{\substack{N(\lambda)\le L}}\cond(\Fc_\lambda)$.
  \item\label{item:modular} If $d\ge 1$, $C\left(L,(\Fc_\lambda)_{\lambda\in\Lambda}\right)\ll dL^{\dim G} \max\limits_{\substack{N(\lambda)\le L}}\max\limits_{M\le N(\lambda)M_G}\sigma_c(X,\Fc_\lambda^{\otimes M})^2$, with $M_G=\rank(G)(\rank(G)+1)/2$.
  \item\label{item:compSystem} If the representations \eqref{eq:rholambda} arise from a compatible system $\rho:\pi_1(X,\overline\eta)\to\GL_r(\prod_{\lambda\in\Lambda}\Oc_\lambda)$, and $X$ has a compactification where it is the complement of a divisor with normal crossing, then
    \[C\left(L,(\Fc_\lambda)_{\lambda\in\Lambda}\right)\ll L^{\dim G+1}r^{\delta_{G=\SL_r}}\left(r+C(X,\rho_{\lambda_0})\right),\]
    where $C(X,\rho_{\lambda_0})$ only depends on $X$ and $\rho_{\lambda_0}$ for an arbitrary fixed $\lambda_0\in\Lambda$.
  \end{enumerate}
\end{theorem}
\begin{table}
  \centering
  \begin{tabular}{c||c|c|c|c|c}
    $G$&$\dim G$&$\rank G$&$E_G$&Type&Weyl group\\ \hline
    $\SL_r$&$r^2-1$&$r-1$&$\frac{2r^2+r-3}{2}$&$A_{r-1}$&$\Sf_r$\\
    $\Sp_r$&$\frac{r(r+1)}{2}$&$r/2$&$\frac{r(2r+3)}{4}$&$C_{r/2}$&$W_{r}\le\Sf_r$
  \end{tabular}
  \caption{Reminder of certain invariants for the groups considered.}
\end{table}
\begin{remarks}
  \begin{enumerate}
  \item In the case of curves ($d=1$) with $E=\Q$ and assumption \ref{item:notPG}, this is \cite[Theorem 3.1, Proposition 3.3]{KowLS06} (see also \cite[Section 5, Remark 5.4]{Kow08}).
  \item We handle the weaker assumption \ref{item:PG} on projective monodromy groups to treat $L$-function attached to even Dirichlet characters over function fields (Section \ref{sec:variancePPSI}).
  \item The constant $C\left(L,(\Fc_\lambda)_{\lambda\in\Lambda}\right)$ may depend on the characteristic $p$, but crucially not on the index $[\F_q:\F_p]$.
  \item The last part of \cite[Remark 5.2]{KowLS06} does not seem quite correct: one crucially has to control the dependency of $C$ with respect to $L$ (that is, the Betti numbers) if one wants to take $L\to\infty$.
  \end{enumerate}
\end{remarks}
In practice, we will use the following consequence of Theorem \ref{thm:largeSieve}:
\begin{corollary}\label{cor:largeSieve}
  In the setting of Theorem \ref{thm:largeSieve}:
  \begin{enumerate}[(a),leftmargin=*]
  \item\label{item:largeSievea} If $X$ is a curve, then
    \[P \Big(q,(\Fc_\lambda,\bs\Omega_\lambda)_{\lambda\in\Lambda}\Big)\ll \frac{t \sup_{\lambda\in\Lambda}  \cond(\Fc_{\lambda})^t}{(1-\delta_{\bs\Omega})\delta_\Lambda}\frac{\log{q}}{q^{1/(tE_G)}},\]
    where the implied constant is absolute and $E_G=\dim G+(\rank{G})/2$.
  \item\label{item:largeSieveb} If there are constants $B_1>0$ and $B_2> 1$ such that
    \[\sup_{\lambda\in\Lambda}\sigma_c(X,\Fc_\lambda^{\otimes N})\le B_1 B_2^N\text{ for all }N\ge 1,\quad\mathrm{ then}\]
    \[P \Big(q,(\Fc_\lambda,\bs\Omega_\lambda)_{\lambda\in\Lambda}\Big)\ll \frac{t^2(B_1^2dr^{\delta_{G=\SL_r}})^t(\log(B_2)M_G+\dim{G})}{(1-\delta_{\bs\Omega})\delta_\Lambda}\frac{\log\log{q}}{\log{q}},\]
    with an absolute implied constant.
  \item\label{item:largeSievec} If hypothesis \ref{item:compSystem} of Theorem \ref{thm:largeSieve} holds, then
      \[P \Big(q,(\Fc_\lambda,\bs\Omega_\lambda)_{\lambda\in\Lambda}\Big)\ll \frac{t \left(r^{\delta_{G=\SL_r}+1}C(X,\rho_{\lambda_0})\right)^t}{(1-\delta_{\bs\Omega})\delta_\Lambda}\frac{\log{q}}{q^{1/(2t(\dim G+1))}}.\]
  \end{enumerate}
\end{corollary}

We prove the theorem and its corollary in the next sections.

\subsection{Preliminaries to the proof of Theorem \ref{thm:largeSieve}\ref{item:modular}}

\subsubsection{Irreducibles in tensor powers of faithful representations}

A classical theorem of Burnside asserts that
\begin{quote}
  if $G$ is a finite group with a faithful (complex) representation $\rho$, then any irreducible representation of $G$ appears as a direct summand of $\rho^{\otimes M}$ for some integer $M\ge 1$ 
\end{quote}
(see e.g. \cite{Ste62,Brau64,BryKov72}). The same result holds for compact groups, and is the key to get bounds on Betti numbers in \cite{Katz16}. For classical groups, this can actually directly be seen from Weyl's constructions of the irreducible modules.

A key input to the proof of Theorem \ref{thm:largeSieve}\ref{item:modular} is the following modular version of Burnside's result, for classical finite groups in defining characteristic.
\begin{proposition}\label{prop:BurnsideLieModular}
  Let $k$ be a field of characteristic $\ell$ and $G=\SL_n(\F_\ell)$ or $\Sp_n(\F_\ell)$ with its standard $k$-representation $\Std: G\to\GL_n(k)$. Any irreducible $k$-representation of $G$ appears as a composition factor\footnote{We need to look at composition factors instead of summands, since we consider modular representations, which are not completely reducible.} of $\Std^{\otimes M}$ for some $M\le \ell M_G$, $M_G=\rank(G)(\rank(G)+1)/2$. Therefore, for any $k$-representation $\pi$ of $G$, the semisimplification $\pi^\semisimple$ appears as a direct summand of $(\dim\pi)(\Std^{\otimes M})^\semisimple$.
\end{proposition}
\begin{proof}
  Since $G$ is defined over $\F_\ell$, any irreducible $k$-representation of $G$ is absolutely irreducible, because $\F_\ell$ is the splitting field of $G$ by a 1968 result of Steinberg \cite[Section 5.2]{Hum06}.

  By a 1963 lifting theorem of Steinberg (see \cite[Section 2.11]{Hum06}), the absolutely irreducible representations of $G$ in characteristic $\ell$ are given by the modules $L(\lambda)$ with $\lambda$ an $\ell$-restricted highest weight, i.e. $0\le \langle\lambda,\alpha^\vee\rangle<q$ for all $\alpha\in\Delta$. For $\omega_i$ ($1\le i\le \rank(G)$) the fundamental dominant weights, that means that $\lambda=\sum_{i=1}^{\rank(G)}a_i\omega_i$ with $0\le a_i<\ell$.

  In Bourbaki numbering \cite[Tables]{BourLie05}, $\omega_i$ is $\Lambda^i(\Std)$ (see ibid, VIII.13.1.IV) (resp. $\ker(\Lambda^i(\Std)\to\Lambda^{i-2}(\Std))$; see ibid, VIII.13.3.IV) for $\SL_n$ (resp. $\Sp_n$). These are simple quotients or subrepresentations of $\Std^{\otimes i}$, so they appear in the composition series.
  %Or Fulton-Harris p. 221 (SL) / Theorem 17.5 (Sp)
\end{proof}
\begin{remark}\label{rem:BurnsideImprove}
  For complex representations, combining David Speyer's proof of Burnside's theorem in \cite{MOSpe11} with character bounds \cite{Glu93} shows that $M\ll \dim G$ is enough, as $\ell\to\infty$. Such an improvement (or even $M\ll\log|\F_\ell|$) to Proposition \ref{prop:BurnsideLieModular} would lead to bounds of the quality of Corollary \ref{cor:largeSieve}\ref{item:largeSievec} in Corollary \ref{cor:largeSieve}\ref{item:largeSieveb}. However, while Brauer characters control composition factors, they do not satisfy (in defining characteristic) good bounds, to extend this characteristic $0$ idea.
\end{remark}

\subsubsection{Betti numbers of reductions modulo $\lambda$ and semisimplifications}
\begin{lemma}\label{lemma:BettiRed}
  In the setting of Theorem \ref{thm:largeSieve}, if $\Fc_\lambda$ is the sheaf of $\F_\lambda$-modules on $X$ obtained by reduction of a lisse sheaf of $\Oc_\lambda$-modules $\hat\Fc_\lambda$ on $X$, then
  \[\sigma_c(X,\hat\Fc_\lambda^{\otimes M})\le \sigma_c(X,\Fc_\lambda^{\otimes M})\le 2\sigma_c(X,\hat\Fc_\lambda^{\otimes M})\]
  for any $M\ge 1$.
\end{lemma}
\begin{proof}
  Let $\Gc=\Fc_\lambda^{\otimes M}$ and $\hat\Gc_\lambda=\hat\Fc_\lambda^{\otimes M}$. The lower bound appears in \cite[p. 279]{KatzSarnak91}, and the same argument yields the upper bound: we have the universal coefficients short exact sequence
  \[0\rightarrow H^i_c(X,\hat\Gc_\lambda)\otimes_{\Oc_\lambda}\F_\lambda \rightarrow H^i_c(X,\Gc_\lambda)\otimes_{\Oc_\lambda}\F_\lambda \rightarrow H^{i+1}_c(X,\hat\Gc_\lambda)[\lambda]\rightarrow 0,\]
  obtained after truncating the long exact sequence in cohomology \cite[1.6.5]{DelEC} associated to the short exact sequence $0\rightarrow \hat\Fc_\lambda\xrightarrow{\cdot \lambda}\hat\Fc_\lambda\rightarrow\Fc_\lambda\rightarrow 0$ . Taking dimensions, this implies that
  \[\sigma_c(X,\hat\Gc_\lambda)\le \sigma_c(X,\Gc_\lambda)\le \sum_{i\ge 0}\left(\dim H^i_c(X,\hat \Gc_\lambda) + \dim H^{i+1}_c(X,\hat\Gc_\lambda)\right).\]
\end{proof}
\begin{remark}
  If the sheaves $\Fc_\lambda$ in Theorem \ref{thm:largeSieve} are obtained by reduction of sheaves of $\Oc_\lambda$-modules $\hat\Fc_\lambda$, Lemma \ref{lemma:BettiRed} shows that it suffices to check hypothesis in \ref{item:largeSieveb} of Corollary \ref{cor:largeSieve} for $\Fc_\lambda$, up to replacing $B_1$ by $2B_1$.
\end{remark}
To deal with non-completely reducible representations, we observe the following:
\begin{lemma}\label{lemma:semisimple}
  Let $\Fc$ be a sheaf of $\F_\ell$-modules on $X$ with composition series
  \[0=\Fc_0\subset\dots\subset\Fc_n=\Fc, \quad \Gc_i:=\Fc_{i+1}/\Fc_i\text{ simple} \quad (0\le i\le n-1).\]
  Then $\sigma_c(X,\Fc^\semisimple)=\sum_{i=0}^{n-1}\sigma_c(X,\Gc_i)=\sigma_c(X,\Fc)$.
\end{lemma}
\begin{proof}
    For all $0\le i\le n-1$, we have a short exact sequence $0\to\Fc_i\to\Fc_{i+1}\to\Gc_i\to 0$, which gives for all $a\ge 0$ a long exact sequence in cohomology
    \[\dots\to H_c^a(X,\Fc_i)\to H_c^a(X,\Fc_{i+1})\to H_c^a(X,\Gc_i)\to H^{a+1}_c(X,\Fc_i)\to\dots\]
    that yields $\sigma_c(X,\Fc_{i+1})=\sigma_c(X,\Gc_i)+\sigma_c(X,\Fc_{i})$, whence $\sigma_c(X,\Fc)=\sigma_c(X,\Fc_n)=\sum_{i=0}^{n-1}\sigma_c(X,\Gc_i)=\sigma_c(X,\Fc^\semisimple)$.
  \end{proof}
  \subsection{Proof of Theorem \ref{thm:largeSieve}}
  We first give the proof under assumption \ref{item:notPG}, before indicating the changes required in the projective case (assumption \ref{item:PG}).\\
  
  For $\lambda,\lambda'\in\Lambda$, we will denote by $\ell,\ell'$ the primes above which they respectively lie. Since $\Lambda\subset\Spec_{1,p}(\Oc)$, note that $\F_\lambda=\F_\ell$, $\F_{\lambda'}=\F_{\ell'}$. We also let $\widehat{G(\F_\ell)}$ be the set of irreducible (complex) representations of $G(\F_\ell)$.

  For every $\lambda\in\Lambda$, we consider the lisse sheaf $\Gc_\lambda=\Fc_\lambda^{\boxtimes t}$ on $X^t$. By \cite[Lemma 5.1]{Kow08}, the natural map $\pi_1(X^t,(\overline\eta,\dots,\overline\eta))\to\pi_1(X,\overline\eta)^t$ is surjective, so that the arithmetic and geometric monodromy groups of $\Gc_\lambda$ are equal and conjugate to $G(\F_\lambda)^t$.
  
  Exactly as in \cite[Theorem 3.1, Proposition 3.3, Section 5]{KowLS06}, we get that
  \begin{eqnarray*}
    P \Big(q,(\Fc_\lambda,\bs\Omega_\lambda)_{\lambda\in\Lambda}\Big)&\ll&\Delta \left[\sum_{\substack{\lambda\in\Lambda \\ N(\lambda)\le L}}\left(1-\frac{|\bs\Omega_\lambda|}{|G(\F_\lambda)|}\right)\right]^{-1}\ll\frac{\Delta\log{L}}{\delta_\Lambda(1-\delta_{\bs\Omega}) L},
  \end{eqnarray*}
  where $\Delta\ll 1+q^{-1/2}\tilde C(L,(\Fc_\lambda)_{\lambda\in\Lambda})$, and $\tilde C(L,(\Fc_\lambda)_{\lambda\in\Lambda})$ is defined by
  \[\max_{\substack{\lambda\in\Lambda \\ N(\lambda)\le L}}\max_{\substack{\bs\pi\in \widehat{G(\F_\lambda)^t} \\ \bs\pi\neq 1}}\left[\sum_{\substack{\bs\pi'\in \widehat{G(\F_{\lambda})^t} \\ \bs\pi'\neq 1}} \sigma_c(X^t, \Fc_{\bs\pi,\bs\pi'})+\sum_{\substack{\lambda'\in\Lambda\\ N(\lambda')\le L \\ \ell'\neq\ell}} \sum_{\substack{\bs\pi'\in\widehat{G(\F_{\lambda'})^t} \\ \bs\pi'\neq 1}} \sigma_c( X^t, \Fc_{\bs\pi,\bs\pi'})\right],\]
  with (see \cite[Proof of Proposition 5.1]{KowLS06})
    \[\Fc_{\bs\pi,\bs\pi'}=
      \tau_{\bs\pi,\bs\pi'}\circ\begin{cases}
        \rho_{\lambda}^{\boxtimes t}&:\ell=\ell'\\
        (\rho_\lambda^{\boxtimes t},\rho_{\lambda'}^{\boxtimes t})&:\ell\neq\ell'
      \end{cases},\quad \tau_{\bs\pi,\bs\pi'}=
      \begin{cases}
        \bs\pi\otimes D(\bs\pi')&:\ell=\ell'\\
        \bs\pi\boxtimes D(\bs\pi')&:\ell\neq\ell',
      \end{cases}
    \]
    identifying lisse sheaves of $\overline\Q_\ell$-modules on $X^t$ and continuous representations $\pi_1(X^t,\overline\eta)\to\GL_m(\overline\Q_\ell)$.
    Note that $\rho_\lambda$ and $(\rho_\lambda,\rho_{\lambda'})$ respectively correspond to sheaves of $\F_\ell$- and $\Z/\ell\ell'$-modules (if $\ell\neq\ell'$).

    Hence, we need to show that
    \[\tilde C(L,(\Fc_\lambda)_{\lambda\in\Lambda})\ll tC(L,(\Fc_\lambda)_{\lambda\in\Lambda})^t,\]
    with $C$ defined in the statement of the theorem. Künneth's formula \cite[Exposé 6, 2.4]{DelEC} reduces this to the case $t=1$.\\

    \subsubsection{Case \ref{item:curves}: curves} The first bound on $C(L,(\Fc_\lambda)_{\lambda\in\Lambda})$ in Theorem \ref{thm:largeSieve}, when $d=1$, is contained in \cite{KowLS06} (with a power of $L$ smaller by one here, because we assume that the arithmetic and geometric monodromy groups coincide).\\

    \subsubsection{Case \ref{item:compSystem}: compatible systems on varieties by reduction to the tame case}\label{subsubsec:compSystemProof}
    Let $\lambda_0\in\Lambda$ be fixed and let $\varphi: Y\to X$ be the étale covering corresponding to $f\pmod{\lambda_0}$. As in \cite[Proposition 4.7]{KowLS06}, by the Hochschild--Serre sequence,
    \[\sigma_c(X, \Fc_{\pi,\pi'})\le\sigma_c(Y,\varphi^*\Fc_{\pi,\pi'}).\]
    It then suffices to show that the compatible system $\rho$ is tame when restricted to $Y$. Indeed, a result of Deligne \cite[Corollaire 2.8]{Ill81} shows that the Euler characteristic of a lisse tame sheaf is equal to its rank times the Euler characteristic of the variety, so by \cite[$\sigma-\chi$ inequality, p.40]{Katz01}, we have in this case
    \begin{eqnarray*}
      \sigma_c(Y,\varphi^*\Fc_{\pi,\pi'})&\ll& r+|\chi_c(Y,\varphi^*\Fc_{\pi,\pi'})|+\sum_{j=1}^{\dim X}|\chi_c(\text{codim }j\text{ in }Y,\varphi^*\Fc_{\pi,\pi'})|\\
                                         &\le&r+\dim(\pi)\dim(\pi')C(X,\rho_{\lambda_0}),
    \end{eqnarray*}
    where $C(X,\rho_{\lambda_0})$ is a constant depending only on the Euler characteristics $\chi_c$ of $Y$ and its subvarieties, hence only on $X$ and $\Fc_{\lambda_0}$.
    Therefore, $\tilde C(L,(\Fc_\lambda)_{\lambda\in\Lambda})$ is
      \begin{eqnarray*}
        &\ll&C(X,\rho_{\lambda_0})\cdot r\max_{\substack{\lambda\in\Lambda \\ N(\lambda)\le L}}\max_{\substack{\pi\in \widehat{G(\F_\lambda)} \\ \pi\neq 1}}d_\pi\left[\sum_{\substack{\pi'\in \widehat{G(\F_{\lambda})} \\ \pi'\neq 1}} d_{\pi'}+\sum_{\substack{\lambda'\in\Lambda\\ N(\lambda')\le L \\ \ell'\neq\ell}} \sum_{\substack{\pi'\in\widehat{G(\F_{\lambda'})} \\ \pi'\neq 1}} d_{\pi'}\right]\\
        &\ll&C(X,\rho_{\lambda_0})\cdot r^{\delta_{G=\SL_r}+1}L^{\dim G+1},
      \end{eqnarray*}
      where $d_\pi:=\dim \pi$. Indeed, the number (complex) of irreducible representations of $G(\F_\ell)$ is given by $|G(\F_\ell)^\sharp|\ll |Z(G(\F_\ell))|\ell^{\rank G}\le r^{\delta_{G=\SL_r}}\ell^{\rank G}$ (see \cite[Corollary 26.10]{TesMal11}), and the maximal dimension of such a representation is $\ll \ell^{\frac{\dim G-\rank G}{2}}$ (see \cite[Proposition 5.4]{KowLargeSieve08}).
      
      To show the tameness of the compatible system restricted to $Y$, first note that it is tame at $\lambda_0$, since it factors by construction through the pro-$\ell_0$-group $\{g\in \GL_n(\Oc_{\lambda_0}) : g\equiv 1\pmod{\lambda_0}\}$, where $\ell_0$ is the prime above which $\lambda_0$ lies. By purity, it suffices to look at restriction to curves (see \cite[Section 2.6]{Ill81}, also \cite{KerSchm10}). In this case, \cite[7.5.1]{KatzTwistedL} shows, from a compatibility result of Deligne, that tameness at one prime implies tameness of the whole system.
    
    \subsubsection{Case \ref{item:modular}: varieties through modular representations} Given $\pi\in \widehat{G(\F_\lambda)}$, $\pi'\in \widehat{G(\F_{\lambda'})}$, we need to bound the sums of Betti numbers $\sigma_c(X,\Fc_{\pi,\pi'})$. By \cite[Corollary 75.4]{CurRei06}, $\pi$ (resp. $\pi'$) is defined over the ring of integers of a finite extension $F_\lambda/E_\lambda$ (resp. $F_{\lambda'}/E_{\lambda'}$), say
    \[\pi: G(\F_\lambda)\to \GL_m(\Oc_{F_\lambda}),\qquad \pi': G(\F_{\lambda'})\to \GL_{m'}(\Oc_{F_{\lambda'}}).\]
    By reduction, we obtain
    \[\tilde\pi: G(\F_\lambda)\to \GL_m(k),\qquad \tilde\pi': G(\F_{\lambda'})\to \GL_{m'}(k'),\]
    for the residue fields $k/\F_\lambda$, resp. $k'/\F_{\lambda'}$. Let $\Std_\lambda: G(\F_\lambda)\to \GL_r(\F_\lambda)$ be the standard representation by inclusion.\label{page:Std}\\
    
    We start with the case $\ell=\ell'$, which is easier. We may then assume that $\F_\lambda=\F_{\lambda'}$. By Lemmas \ref{lemma:BettiRed} and \ref{lemma:semisimple}, along with the fact that $\rho_\lambda: \pi_1(X,\overline\eta)\to G(\F_\lambda)$ is surjective,
    \[\sigma_c(X, \Fc_{\pi,\pi'})\le \sigma_c(X, \Fc_{\tilde\pi,\tilde\pi'})\le\sigma_c(X,\tau_{\tilde\pi,\tilde\pi'}^\semisimple\circ\rho_\lambda).\]
    By Proposition \ref{prop:BurnsideLieModular}, every simple summand of $\tau_{\tilde\pi,\tilde\pi'}^\semisimple$ appears as a composition factor of $(\Std_\lambda\otimes\overline\F_\ell)^{\otimes M}$ for some $M\le \ell M_G$. It follows that
    \begin{eqnarray}      
      \sigma_c(X, \Fc_{\pi,\pi'})&\le&(\dim\pi)(\dim\pi')\max_{M\le \ell M_G}\sigma_c\left(X,\Fc_\lambda^{\otimes M}\right).\label{eq:conclell}
    \end{eqnarray}                
  
    Let us now assume that $\ell\neq\ell'$, and note that $(\rho_\lambda,\rho_{\lambda'})$ corresponds to the sheaf of $\Z/\ell\ell'$-modules on $X$ given by $\Delta^*(\Fc_\lambda\boxtimes\Fc_{\lambda'})$, for $\Delta: X\to X\times X$ the diagonal immersion. We may view $\Fc_{\pi,\pi'}$ as sheaf of $(\Oc_{F_\lambda}\otimes\Oc_{F_{\lambda'}})$-modules, and $\sigma_c(X,\Fc_{\pi,\pi'})$ is equal to the sum of the ranks (under Definition \ref{def:Betti}) of the corresponding étale cohomology groups with compact support. Then $\Fc_{\tilde\pi,\tilde\pi'}$ is a sheaf of $(k\otimes k')$-modules, and by Lemma \ref{lemma:semisimple} and the same argument as in Lemma \ref{lemma:BettiRed},
    \[\sigma_c(X,\Fc_{\pi,\pi'})\le \sigma_c(X,\Fc_{\tilde\pi,\tilde\pi'})=\sigma_c \left(X,\tau_{\tilde\pi,\tilde\pi'}^\semisimple\circ\Delta^*(\Fc_\lambda\boxtimes\Fc_{\lambda'})\right).\]
    As above, we get that every simple summand in $\tau_{\tilde\pi,\tilde\pi'}^\semisimple$ appears as a composition factor of the $(k\otimes k')$-module $\Std^{\otimes M}\boxtimes\Std^{\otimes M'}$ for some $M\le \ell M_{G}$ and $M'\le \ell' M_{G}$. This implies that
    \[\sigma_c(X,\Fc_{\tilde\pi,\tilde\pi'})\le (\dim\pi+\dim\pi')\max_{M\le \ell M_G}\max_{M'\le \ell' M_G}\sigma_c \left(X, \Delta^*\Gc_{M,M'}\right),\]
    where $\Gc_{M,M'}=(\Fc_\lambda\otimes k)^{\otimes M}\boxtimes(\Fc_{\lambda'} \otimes k')^{\otimes M'}$.

    By purity \cite[Corollary 8.5.6]{Fu11} and the localization sequence \cite[Proposition 5.6.11]{Fu11}, this implies that $\sigma_c \left(X, \Delta^*\Gc_{M,M'}\right)\le \sigma_c \left(X\times X, \Gc_{M,M'}\right)$.
    By Künneth's formula \cite[Exposé 6, 2.4]{DelEC},
    \begin{eqnarray*}
      \rank H^i_c(X\times X, \Gc_{M,M'})&=&\sum_{a+b=i}\rank H^a_c(X, \Fc_\lambda^{\otimes M})\rank H^b_c(X, \Fc_{\lambda'}^{\otimes M'})\\
                                 &\le&\sigma_c \left(X,\Fc_\lambda^{\otimes M}\right)\sigma_c \left(X,\Fc_{\lambda'}^{\otimes M'}\right),
    \end{eqnarray*}
    hence
      \begin{equation}
        \label{eq:conclellell'}
              \sigma_c(X,\Fc_{\pi,\pi'})\ll d(\dim\pi+\dim\pi')S(\lambda)S(\lambda')
            \end{equation}
            where $S(\lambda):=\max_{M\le N(\lambda)M_G} \sigma_c(X,\Fc_\lambda^{\otimes M})$.
      
      Thus, \eqref{eq:conclell} and \eqref{eq:conclellell'} yield that $\tilde C(L,(\Fc_\lambda)_{\lambda\in\Lambda})$ is, as in Section \ref{subsubsec:compSystemProof},
    \begin{eqnarray*}
       &\ll&\max_{\substack{\lambda\in\Lambda\\ N(\lambda)\le L}}S(\lambda) \max_{\substack{\pi\in \widehat{G(\F_\ell)} \\ \pi\neq 1}}\left[d_\pi \sum_{\substack{\pi'\in \widehat{G(\F_{\ell})} \\ \pi'\neq 1}} d_{\pi'}+d\sum_{\substack{\lambda'\in\Lambda\\N(\lambda')\le L \\ \ell'\neq\ell}} \sum_{\substack{\pi'\in\widehat{G(\F_{\ell'})} \\ \pi'\neq 1}} (d_\pi+d_{\pi'})S(\lambda')\right]\\
       &\ll&dr^{\delta_{G=\SL_r}}L^{\dim G}\max_{N(\lambda)\le L}S(\lambda)^2.
    \end{eqnarray*}
 
    \subsubsection{Projective monodromy groups}
    Let us now suppose that only assumption \ref{item:PG} holds. For $\eta: G\to PG$ the projection, we have
    \[P \Big(q,(\Fc_\lambda,\bs\Omega_\lambda)_{\lambda\in\Lambda}\Big)\le \frac{|\{\bs x\in X(\F_q)^t : (\eta\rho_\lambda(\Frob_{x_i,q}))_i\in\eta(\bs\Omega_\lambda)\text{ for all }\lambda\in\Lambda\}|}{|X(\F_q)|^t},\]
      and for any $\Omega\subset G(\F_\lambda)$,
      \[\frac{|\eta(\Omega)|}{|PG(\F_\lambda)|}=\frac{|\eta(\Omega)||Z(G(\F_\lambda))|}{|G(\F_\lambda)|}\le \frac{|\Omega|}{|G(\F_\lambda)|}=\delta_\Omega.\]
      Thus, it is enough to repeat the arguments above with $G$ replaced by $PG$. Indeed, since $(r,|\F_\lambda|-1)=r$ for all $\lambda\in\Lambda$, we have $\SL_r(\F_\lambda)\cong\PGL_r(\F_\lambda)$, so this can be done mutatis mutandis (in particular, the ``standard representation'' $PG(\F_\lambda)\to \GL_r(\F_\lambda)$ on page \pageref{page:Std} is well-defined).
    \qed

\section{Generic maximality of splitting fields and linear independence}\label{sec:genericMax}
This section mostly recalls some results from \cite{Kow08} and gives their analogues for $\SL$ when necessary.
\subsection{Generic maximality of splitting fields}
\begin{definition}
  For $R$ a ring and $r\ge 2$ an integer, we let
  \begin{eqnarray*}
    \Pc_{\SL_r}(R)&:=&\{P\in R[T]\text{ monic} : \deg(P)=r, \ P(0)=1\} \quad (r\ge 2),\\
    \Pc_{\Sp_r}(R)&:=&\{P\in \Pc_{\SL_r}(R) : P(T)=T^rP(1/T)\} \quad (r\ge 2\text{ even}).
  \end{eqnarray*}
\end{definition}
Note that for $G\in\{\SL_r,\Sp_r\}$, the set of (reversed) characteristic polynomials of elements of $G(R)$ is included in $\Pc_G(R)$, with equality at least when $R$ is a finite field (see the reference to Chavdarov's proof in \cite[Lemma B.5(2)]{KowLargeSieve08}).

Let $E$ be a Galois number field with ring of integers $\Oc$. Note that the Galois group of a polynomial $P\in \Pc_{G}(\overline E)$ of degree $n$ is contained in 
\begin{itemize}
\item $\Sf_r$ if $G=\SL_r$.
\item $W_{r}\le \Sf_r$ (the Coxeter group $B_{r/2}$) if $G=\Sp_r$ ($r$ even).
\end{itemize}
We will say that the Galois group is \emph{non-maximal} if this inclusion is strict.

\subsubsection{Detecting non-maximal Galois groups}

\begin{proposition}\label{prop:detectNonmax}
   Let $G=\SL_r$ \textup{(}$r\ge 2$\textup{)} or $G=\Sp_r$ \textup{(}$r\ge 2$ even\textup{)}. For every $t\ge 1$ and $\lambda\in\Spec_{1}(\Oc)$, there exist conjugacy-invariant sets $\bs \Omega_{\bs i,\lambda,G^t}\subset G(\F_\lambda)^t$ \textup{(}$\bs i\in \bs I$, with $\bs I$ an index set of size $4t$\textup{)} such that:
  \begin{itemize}[itemsep=0.2cm]
  \item $\bs\Omega_{i,\lambda,G^t}$ has density $\le \delta_{r,t}:=\left(\left(1-\frac{1}{r!}\right)\left(1+\frac{r}{\ell}\right)\right)^t\left(1-\frac{1}{2r}\right)$.
  \item If $\bs g=(g_1,\dots,g_t)\in \Pc_G(\Oc_\lambda)^t$ is such that
    $\prod_{i=1}^t \det(1-Tg_i)\in \Pc_G(\Oc_\lambda)\subset\Oc_\lambda[T]$
    has non-maximal Galois group, that is, strictly contained in $\Sf_r^t$ (resp. $W_r^t$) if $G=\SL_r$ (resp. $\Sp_r$), then there exists $\bs i\in \bs I$ such that $\bs g\pmod*{\lambda}\in\bs \Omega_{\bs i,\lambda,G^t}$.
  \end{itemize}  
\end{proposition}
\begin{proof}
  The case $G=\Sp_r$ is contained in \cite[Proof of Theorem 4.3]{Kow08} (see also \cite[Proof of Theorem 8.13]{KowLargeSieve08}), using \cite[Lemma B.5]{KowLargeSieve08} to switch between densities of matrices and characteristic polynomials, and up to replacing $\Z$ by $\Oc_\lambda$.

  The case $G=\SL_r$ is simpler, and we also apply the lemma of Bauer quoted by Gallagher \cite[p. 98]{Gall73}: if $H\le \Sf_r$ is transitive, contains a transposition and a $m$-cycle with $m>r/2$ prime, then $H=\Sf_r$. We define
  \begin{eqnarray*}
    \tilde\Omega_{0,\lambda}&=&\{P\in\Pc_{\SL_r}(\F_\lambda) : \text{product of linear factors}\}^c,\\
    \tilde\Omega_{1,\lambda}&=&\{P\in\Pc_{\SL_r}(\F_\lambda) : P\text{ reducible}\},\\
    \tilde\Omega_{2,\lambda}&=&\{P\in\Pc_{\SL_r}(\F_\lambda) : P=QQ_1\dots Q_s, \ \substack{Q,Q_i\text{ irred}\\\deg(Q)=2, \ \deg(Q_i)\text{ odd}}\}^c,\\
    \tilde\Omega_{3,\lambda}&=&\{P\in\Pc_{\SL_r}(\F_\lambda) : P \text{ has }\substack{\text{an irreducible factor}\\\text{of prime degree} }>r/2\}^c,\\
    \Omega_{j,\lambda}&=&\{g\in\SL_r(\F_\lambda) : \det(1-Tg)\in\tilde\Omega_{j,\lambda}\} \quad(0\le j\le 3),\\
    {\bs\Omega}_{\bs i,\lambda,G^t}&=&\Omega_{0,\lambda}^{k-1}\times\Omega_{j,\lambda}\times\Omega_{0,\lambda}^{t-k},\quad \bs i=(k,j)\in \bs I:=\{1,\dots,t\}\times\{1,2,3\},
  \end{eqnarray*}
  (we make the reader attentive to the fact that some of the sets above are defined using complements) and the same arguments as in the $\Sp_r$ case give the conclusion..
\end{proof}
\subsubsection{Application of the large sieve}
\begin{corollary}\label{cor:maxGal}
  Let $X$, $E$, $\Oc$ and $\Lambda$ be as in Theorem \ref{thm:largeSieve}. For every $\lambda\in\Oc$, let $\hat\Fc_\lambda$ be a rank $r$ lisse sheaf of free $\Oc_\lambda$-modules on $X$, corresponding to a representation $\hat\rho_\lambda: \pi_1(X,\overline\eta)\to\GL_r(\Oc_\lambda)$.
  We assume assumption \ref{item:notPG} or \ref{item:PG} of Theorem \ref{thm:largeSieve}, and hypothesis \ref{item:largeSievea}, \ref{item:largeSieveb} or \ref{item:largeSievec} of Corollary \ref{cor:largeSieve}, hold for $\hat\rho_\lambda$.
  For $\bs x\in X(\F_q)^t$, let
  \[P_\lambda(\bs x):= \prod_{i=1}^t P_\lambda(x_i),\quad P_\lambda(x_i)=\det(1-T\rho_\lambda(\Frob_{x_i,q})).\]
  Then, for every $t\ge 1$ and every finite field $\F_q$ of characteristic $p$, we have
  \begin{eqnarray}
    &&\frac{|\{\bs x\in X(\F_q)^t : P_\lambda(\bs x)\in\Oc_\lambda[T]\ \substack{\mathrm{has\, non-maximal}\\ \mathrm{Galois\, group}} \ \forall\lambda\in\Lambda\}|}{|X(\F_q)|^t}\label{eq:cor:maxGal}\\
    &\ll&\frac{t^2}{(1-\delta_{r,t})\delta_\Lambda}\begin{cases}
      \sup_{\lambda\in\Lambda}  \cond(\Fc_{\lambda})^t\frac{\log{q}}{q^{1/(tE_G)}}&\text{under \ref{item:largeSievea}}\\
            t(B_1^2dr^{\delta_{G=\SL_r}})^t(\log(B_2)M_G+\dim{G})\frac{\log\log{q}}{\log{q}}&\text{under \ref{item:largeSieveb}}\\
            \left(r^{\delta_{G=\SL_r}+1}C(X,\rho_{\lambda_0})\right)^t\frac{\log{q}}{q^{1/(2t(\dim G+1))}}&\text{under \ref{item:largeSievec}}\\
          \end{cases}\nonumber
  \end{eqnarray}
  with an absolute implied constant.
\end{corollary}
\begin{proof}
  By Proposition \ref{prop:detectNonmax}, the density on the left-hand side is
  \begin{eqnarray*}
    &\le&\sum_{\bs i\in \bs I}\frac{|\{\bs x\in X(\F_q)^t : (\rho_\lambda(\Frob_{x_i,q}))_{1\le i\le t}\in\bs\Omega_{\bs i,\lambda,G^t} \  \forall\lambda\in\Lambda\}|}{|X(\F_q)|^t},
  \end{eqnarray*}
  and it suffices to apply Corollary \ref{cor:largeSieve} to each summand.
\end{proof}

\subsection{Girstmair's method}
Below, we recall the following forms of Girstmair's results \cite{Girst82,Girst99}, as exposed in \cite{Kow08} (with some changes in the symmetric case).
\begin{definition}
  For a set $M$ of complex numbers, let
  \begin{eqnarray*}
    \Rel_m(M)&=&\{(n_\alpha)\in\Z^M : \prod_{\alpha\in M}\alpha^{n_\alpha}=1\}.
  \end{eqnarray*}
\end{definition}
\begin{proposition}\label{prop:Girstmair}
  Let $E$ be a number field, $t\ge 1$ an integer, and for $1\le i\le t$, let $P_i\in E[X]$ be a polynomial with splitting field $K_i$, set of roots $M_i\subset K_i$, and Galois group $G_i:=\Gal(K_i/E)$. We assume that the fields $K_i$ are linearly disjoint, and we let $M=\bigcup_{i=1}^t M_i$, $K=K_1\cdots K_t$. Then
  $\Rel_m(M)\otimes\Q=\bigoplus_{i=1}^t \Rel_m(M_i)\otimes\Q$. Moreover:
  \begin{enumerate}
  \item \textup{(}$W$ case\textup{)} Assume that $G_i\cong W_{r}$ for some $r\ge 4$ even, acting by permutation on $M_i$.
    If $|\alpha|=1$ for every $\alpha\in M_i$, then
    \[\Rel_m(M_i)\otimes\Q=\left\{(n_\alpha)\in\Q^{M_i} : n_{\alpha}=n_{\overline\alpha}\right\}.\]
  \item \textup{(}$\Sf$ case\textup{)} Assume that $G_i\cong \Sf_r$ for some $r\ge 2$, acting by permutation on $M_i$. Then $\Rel_m(M_i)\otimes\Q$ is either:
    \begin{enumerate}
    \item if $r=2$:\quad $0$,\quad $\Q\bs 1$,\quad or \quad $\Q(-1,1)$.
    \item if $r\ge 3$:\quad $0$\quad or\quad $\Q\bs 1$.
    \end{enumerate}
  \end{enumerate}
  \end{proposition}
  \begin{proof}
    The $W$ case is \cite[Proposition 2.4, (2.5)]{Kow08}. However, $\Q$ in the paragraph after the second display of \cite[p. 13]{Kow08} should probably be replaced by $E$, and the contradiction comes from the fact that the splitting field of $K/E$ would be a 2-group. 

    For the $\Sf$ case, note that the permutation representation $F(M_i)$ of $\Sf_r$ decomposes as the sum of two irreducible representations
    \[F(M_i)=\Q\bs 1\bigoplus G(M_i), \text{ where }  G(M_i)=\left\{(n_\alpha)\in\Q^{M_i} : \sum_{\alpha\in M_i}n_\alpha=0\right\}.\]
    If $G(M_i)$ is contained in the subrepresentation $\Rel_m(M_i)\otimes\Q$ of $F(M_i)$, then there exists $m\ge 1$ such that $(\alpha_j/\alpha_1)^m=1$ for $1\le j\le r$, if $M_i=\{\alpha_1,\dots,\alpha_r\}$, so that $\alpha_1^{nm}=N_{M_i/E}(\alpha_1)^m\in E$. Hence, $K_i/E$ is a Kummer extension and $\Gal(K_i/E)$ is abelian, which implies that $r=|M_i|=2$. If $r=2$, note that $\Rel_m(M_i)\otimes\Q=\Q^2$ would imply that $\Rel_m(M_i)=\Z^2$, a contradiction.
  \end{proof}
  
\subsection{Conclusion}
\begin{corollary}\label{cor:linIndep}
  Under the hypotheses of Corollary \ref{cor:maxGal}, assume moreover that $(\Fc_\lambda)_{\lambda\in\Lambda}$ forms a \emph{compatible system}, i.e. that for all $x\in X(\F_q)$, $P_\lambda(x)=P(x)\in E[T]$ does not depend on $\lambda$. For every $\bs x\in X(\F_q)^t$ and $1\le i\le t$, let $M(x_i)\subset\C$ be the set of zeros of $\det(1-T\rho_\lambda(\Frob_{x_i,q}))$, so that the set of zeros of $P_\lambda(\bs x)$ is $\bigcup_{i=1}^t M(x_i)$. Then, for all but at most a proportion \eqref{eq:cor:maxGal} of $\bs x\in X(\F_q)^t$, we have
  \begin{eqnarray*}
    \Rel_m(M(\bs x))&=&
                        \begin{cases}
                          \otimes_{i=1}^t\Z\bs 1&:G=\SL_r \ (r\ge 2)\\
                          \otimes_{i=1}^t\{(n_\alpha)\in\Z^{M(x_i)}: n_\alpha=n_{\overline\alpha}\}&:G=\Sp_r \ (r\ge 4\text{ even}).
                        \end{cases}
  \end{eqnarray*}
  In other words, the only multiplicative relations among the roots are the trivial ones. If we write the roots of $P(x_i)$ as
  \[
    \begin{cases}
      e(\theta_{j}(x_i)) \quad \ \ (1\le j\le r)&:G=\SL_r\\
      e(\pm \theta_{j}(x_i)) \quad (1\le j\le r/2)&:G=\Sp_r,
    \end{cases}
  \]
  then the angles
  \[
    \begin{cases}
      1, \ \theta_{j}(x_i) \quad (1\le i\le t, \ 1\le j\le r-1)&:G=\SL_r\\
      1, \ \theta_{j}(x_i) \quad (1\le i\le t, \ 1\le j\le r/2)&:G=\Sp_r
    \end{cases}
  \]
  are $\Q$-linearly independent for all but at most a proportion \eqref{eq:cor:maxGal} of $\bs x\in X(\F_q)^t$.
\end{corollary}
\begin{proof}
  By the compatibility assumption and Corollary \ref{cor:maxGal}, $P_\lambda(\bs x)$ has maximal Galois group $\Sf_r^t$ or $W_r^t$ for all but at most a proportion \eqref{eq:cor:maxGal} elements $\bs x\in X(\F_q)^t$. Let us assume this maximality condition holds, in which case the hypotheses of Proposition \ref{prop:Girstmair} hold. Since the product of the zeros of $P_{\lambda}(x_i)$ is equal to $1$, we have $\Z\bs 1\subset\Rel_m(M(x_i))$ for all $x_i\in X(\F_q)$. By Proposition \ref{prop:Girstmair} and the fact that $\Rel_m(M(x_i))$ is a lattice, this implies that $\Rel_m(M(\bs x))$ is as given in the statement.
\end{proof}

\section{Proof of the generic linear independence theorems}\label{sec:proofGenericLI}
In this section, we finally prove Theorems \ref{thm:linIndepExp}, \ref{thm:genericLIGP} and \ref{thm:genericLIGPchi}, by applying Corollary \ref{cor:linIndep}. That basically means checking that assumptions \ref{item:notPG} or \ref{item:PG} of Theorem \ref{thm:largeSieve} (on monodromy groups) apply, as well as hypothesis \ref{item:largeSievea} or \ref{item:largeSievec} of Corollary \ref{cor:largeSieve}.

\subsection{Proof of Theorem \ref{thm:linIndepExp} (exponential sums)}
Assumption \ref{item:notPG} of Theorem \ref{thm:largeSieve} holds by Theorems \ref{thm:monKl} and \ref{thm:monBirch} for Kloosterman sums and Birch sums respectively, with the set of valuations $\Lambda_{r,p}$, $\Lambda_p$ given therein. For Kloosterman sums, the dependency with respect to $p$ can be removed by Remark \ref{rem:monKl}.

Since the sheaves are on curves, \ref{item:largeSievea} of Corollary \ref{cor:largeSieve} holds. By \cite[Theorem 4.1.1(3,4)]{KatzGKM}, $\cond(\Kl_{r,\lambda})$ is bounded by a constant depending only on $r$ (and not on $p$), and the same holds true for Birch sheaves by the bounds on Swan conductors and ramification points in \cite[Chapter 7]{KatzESDE}.
\qed
\subsection{Proof of Theorem \ref{thm:genericLIGP} (super-even primitive characters)}
Assumption \ref{item:notPG} of Theorem \ref{thm:largeSieve} applies by Theorem \ref{thm:monEvenDir}, with the set of valuations $\Lambda_{k,p}$ given by the latter.

If $p>k$, we see (as in \cite[Lemma 5.2]{Katz16}) that $\W_{2\kappa\codd}=\prod_{1\le a\le 2\kappa,\, a\text{ odd}}W_1$ is the space of odd polynomials of degree $\le 2\kappa-1$ and $\Prim_{2\kappa\codd}$ the subspace of those polynomials with degree exactly $2\kappa-1$. One can then apply Corollary \ref{cor:largeSieve}\ref{item:largeSievec}, which gives the theorem.\\

To obtain the weaker error (but with explicit base of $t$) in Remark \ref{rem:weakExplicit1}, one applies Corollary \ref{cor:largeSieve}\ref{item:largeSieveb} instead, using the bounds for Betti numbers in \cite[Lemma 5.2]{Katz16}, giving $B_1=3(2\kappa+1)^{2\kappa}$, $B_2=2\kappa+1$
\qed
\subsection{Proof of Theorem \ref{thm:genericLIGPchi} (even primitive characters)}
In this case, hypothesis \ref{item:PG} of Theorem \ref{thm:largeSieve} (projective monodromy groups) applies by Theorem \ref{thm:monEvenDir}.

If $p>m$, then as in \cite{Katz16}, we see that $\W_m=\prod_{1\le a\le m}W_1$ is the space of polynomials of degree $\le m$ with constant term $1$ and $\Prim_m$ is the subspace of those polynomials with degree exactly $m$. One can then apply Corollary \ref{cor:largeSieve}\ref{item:largeSievec}, which gives the theorem.\\

To obtain the weaker error (but with explicit base of $t$) in Remark \ref{rem:weakExplicit2}, one applies Corollary \ref{cor:largeSieve}\ref{item:largeSievec} instead, proceeding from \cite{Kat13} as in \cite[Lemma 5.2]{Katz16} to bound the Betti numbers. Let us indeed show that Hypothesis \ref{item:largeSieveb} of Corollary \ref{cor:largeSieve} holds with $B_1=3(m+1)^{m+1}$ and $B_2=m+1$. Let $M\ge 1$ be an integer. With coordinates $(t_1,\dots,t_M,f)$ on $\A^M\times\Prim_m$,
  \[H^i_c\left(\Prim_m,\Lc_\univ^{\otimes M}\right)=H^{i+M}_c\left(\A^M\times\Prim_{m}, \Lc_{\psi(f(t_1)+\dots+f(t_M))}\right).\]
  Note that $\A^M\times\Prim_m$ is defined in $\A^{M+1+m}$ (an additional coordinate is needed for the condition that $a_m\neq 0$) and $f(t_1)+\dots+f(t_M)$ is a polynomial in $t_i,a_i$ of degree $m+1$. By \cite[Theorem 12]{Katz01} (with $(\delta,N,r,d,s,e_j)=(m+1,M+1+m,1,2,0,0)$), we have
  \begin{eqnarray*}
    \sigma_c\left(\Prim_m,\Lc_\univ^{\otimes M}\right)&\le&3 \left(1+\max(m+1,3)\right)^{M+m+1}=3 \left(m+1\right)^{M+m+1}.
  \end{eqnarray*}
  \qed

\ifarxiv
\bibliographystyle{alpha2}
\bibliography{references}

\begin{thebibliography}{MVW84}

\bibitem[AS10]{AhmSphar10}
Omran Ahmadi and Igor~E. Shparlinski.
\newblock On the distribution of the number of points on algebraic curves in
  extensions of finite fields.
\newblock {\em Math. Res. Lett.}, 17(4):689--699, 2010.

\bibitem[Bil86]{Bill86}
Patrick Billingsley.
\newblock {\em Probability and measure}.
\newblock John Wiley \& Sons, second edition, 1986.

\bibitem[BK72]{BryKov72}
Roger~M. Bryant and L\'{a}szl\'{o}~G. Kov\'{a}cs.
\newblock Tensor products of representations of finite groups.
\newblock {\em Bull. London Math. Soc.}, 4:133--135, 1972.

\bibitem[BK10]{BombKatz10}
Enrico Bombieri and Nicholas~M. Katz.
\newblock A note on lower bounds for {F}robenius traces.
\newblock {\em Enseign. Math. (2)}, 56(3-4):203--227, 2010.

\bibitem[Bou05]{BourLie05}
Nicolas Bourbaki.
\newblock {\em Lie groups and {L}ie algebras. {C}hapters 7--9}.
\newblock Elements of Mathematics (Berlin). Springer-Verlag, Berlin, 2005.
\newblock Translated from the 1975 and 1982 French originals by Andrew
  Pressley.

\bibitem[Bra64]{Brau64}
Richard Brauer.
\newblock A note on theorems of {B}urnside and {B}lichfeldt.
\newblock {\em Proc. Amer. Math. Soc.}, 15:31--34, 1964.

\bibitem[BW93]{BakWu93}
Alan Baker and Gisbert W\"{u}stholz.
\newblock Logarithmic forms and group varieties.
\newblock {\em J. Reine Angew. Math.}, 442:19--62, 1993.

\bibitem[CFJ16]{ChaFioJou16}
Byungchul Cha, Daniel Fiorilli, and Florent Jouve.
\newblock Prime number races for elliptic curves over function fields.
\newblock {\em Ann. Sci. \'{E}c. Norm. Sup\'{e}r. (4)}, 49(5):1239--1277, 2016.

\bibitem[CFJ17]{ChaFioJou16b}
Byungchul Cha, Daniel Fiorilli, and Florent Jouve.
\newblock Independence of the zeros of elliptic curve {$L$}-functions over
  function fields.
\newblock {\em Int. Math. Res. Not. IMRN}, 2017(9):2614--2661, 2017.

\bibitem[Cha08]{Cha08}
Byungchul Cha.
\newblock Chebyshev's bias in function fields.
\newblock {\em Compos. Math.}, 144(6):1351--1374, 2008.

\bibitem[CK10]{ChaKim10}
Byungchul Cha and Seick Kim.
\newblock Biases in the prime number race of function fields.
\newblock {\em J. Number Theory}, 130(4):1048--1055, 2010.

\bibitem[CR06]{CurRei06}
Charles~W. Curtis and Irving Reiner.
\newblock {\em Representation theory of finite groups and associative
  algebras}.
\newblock AMS Chelsea Publishing, Providence, RI, 2006.
\newblock Reprint of the 1962 original.

\bibitem[Del77]{DelEC}
Pierre Deligne.
\newblock {\em Cohomologie \'{e}tale}, volume 569 of {\em Lecture Notes in
  Mathematics}.
\newblock Springer-Verlag, Berlin, 1977.
\newblock S\'{e}minaire de g\'{e}om\'{e}trie alg\'{e}brique du Bois-Marie SGA
  $4\frac{1}{2}$.

\bibitem[Dev19]{Dev18}
Lucile Devin.
\newblock Chebyshev's bias for analytic {$L$}-functions.
\newblock {\em Math. Proc. Cambridge Philos. Soc.}, 2019.
\newblock To appear.

\bibitem[DM18]{DevMeng18}
Lucile Devin and Xianchang Meng.
\newblock Chebyshev's bias for products of irreducible polynomials.
\newblock 2018.
\newblock Preprint arXiv:1809.09662.

\bibitem[Eve84]{Ev84}
Jan-Hendrik Evertse.
\newblock On sums of {$S$}-units and linear recurrences.
\newblock {\em Compositio Math.}, 53(2):225--244, 1984.

\bibitem[FKM15]{FKMSumProducts}
\'{E}tienne Fouvry, Emmanuel Kowalski, and Philippe Michel.
\newblock A study in sums of products.
\newblock {\em Philos. Trans. Roy. Soc. A}, 373(2040):20140309, 26, 2015.

\bibitem[Fu11]{Fu11}
Lei Fu.
\newblock {\em Etale cohomology theory}, volume~13 of {\em Nankai Tracts in
  Mathematics}.
\newblock World Scientific Publishing Co. Pte. Ltd., Hackensack, NJ, 2011.

\bibitem[Gal73]{Gall73}
Patrick~X. Gallagher.
\newblock {\em The large sieve and probabilistic {G}alois theory}.
\newblock Amer. Math. Soc., Providence, R.I., 1973.

\bibitem[Gir82]{Girst82}
Kurt Girstmair.
\newblock Linear dependence of zeros of polynomials and construction of
  primitive elements.
\newblock {\em Manuscripta Math.}, 39(1):81--97, 1982.

\bibitem[Gir99]{Girst99}
Kurt Girstmair.
\newblock Linear relations between roots of polynomials.
\newblock {\em Acta Arith.}, 89(1):53--96, 1999.

\bibitem[Glu93]{Glu93}
David Gluck.
\newblock Character value estimates for nonsemisimple elements.
\newblock {\em J. Algebra}, 155(1):221--237, 1993.

\bibitem[Gou06]{Gou06}
Nicolas Gouillon.
\newblock Explicit lower bounds for linear forms in two logarithms.
\newblock {\em Journal de th\'eorie des nombres de Bordeaux}, 18(1):125--146,
  2006.

\bibitem[Hal08]{Hall08}
Chris Hall.
\newblock Big symplectic or orthogonal monodromy modulo {$l$}.
\newblock {\em Duke Math. J.}, 141(1):179--203, 2008.

\bibitem[Hum06]{Hum06}
James~E. Humphreys.
\newblock {\em Modular representations of finite groups of {L}ie type}, volume
  326 of {\em London Mathematical Society Lecture Note Series}.
\newblock Cambridge University Press, Cambridge, 2006.

\bibitem[Ill81]{Ill81}
Luc Illusie.
\newblock Th\'{e}orie de {B}rauer et caract\'{e}ristique
  d'{E}uler-{P}oincar\'{e} (d'apr\`es {P}. {D}eligne).
\newblock In {\em The {E}uler-{P}oincar\'{e} characteristic ({F}rench)},
  volume~82 of {\em Ast\'{e}risque}, pages 161--172. Soc. Math. France, Paris,
  1981.

\bibitem[JKZ13]{JKZ13}
Florent Jouve, Emmanuel Kowalski, and David Zywina.
\newblock Splitting fields of characteristic polynomials of random elements in
  arithmetic groups.
\newblock {\em Israel J. Math.}, 193(1):263--307, 2013.

\bibitem[Kat87]{KatzMonodromyFamES}
Nicholas~M. Katz.
\newblock On the monodromy groups attached to certain families of exponential
  sums.
\newblock {\em Duke Math. J.}, 54(1):41--56, 1987.

\bibitem[Kat88]{KatzGKM}
Nicholas~M. Katz.
\newblock {\em Gauss sums, {K}loosterman sums, and monodromy groups}, volume
  116 of {\em Annals of Mathematics Studies}.
\newblock Princeton University Press, Princeton, NJ, 1988.

\bibitem[Kat90]{KatzESDE}
Nicholas~M. Katz.
\newblock {\em Exponential sums and differential equations}, volume 124 of {\em
  Annals of Mathematics Studies}.
\newblock Princeton University Press, Princeton, NJ, 1990.

\bibitem[Kat01]{Katz01}
Nicholas~M. Katz.
\newblock Sums of {B}etti numbers in arbitrary characteristic.
\newblock {\em Finite Fields Appl.}, 7(1):29--44, 2001.
\newblock Dedicated to Professor Chao Ko on the occasion of his 90th birthday.

\bibitem[Kat02]{KatzTwistedL}
Nicholas~M. Katz.
\newblock {\em Twisted {$L$}-functions and monodromy}, volume 150 of {\em
  Annals of Mathematics Studies}.
\newblock Princeton University Press, Princeton, NJ, 2002.

\bibitem[Kat05]{KatzMMP}
Nicholas~M. Katz.
\newblock {\em Moments, monodromy, and perversity: a {D}iophantine
  perspective}, volume 159 of {\em Annals of Mathematics Studies}.
\newblock Princeton University Press, Princeton, NJ, 2005.

\bibitem[Kat12a]{KatzMellin12}
Nicholas~M. Katz.
\newblock {\em Convolution and equidistribution}, volume 180 of {\em Annals of
  Mathematics Studies}.
\newblock Princeton University Press, Princeton, NJ, 2012.
\newblock Sato-Tate theorems for finite-field Mellin transforms.

\bibitem[Kat12b]{Katz12}
Nicholas~M. Katz.
\newblock Report on the irreducibility of {$L$}-functions.
\newblock In Dorian Goldfeld, Jay Jorgenson, Peter Jones, Dinakar Ramakrishnan,
  Kenneth Ribet, and John Tate, editors, {\em Number Theory, Analysis and
  Geometry}, pages 321--353. Springer, 2012.

\bibitem[Kat13a]{Kat13Primitive}
Nicholas~M. Katz.
\newblock On a question of {K}eating and {R}udnick about primitive {D}irichlet
  characters with squarefree conductor.
\newblock {\em Int. Math. Res. Not. IMRN}, 2013(14):3221--3249, 2013.

\bibitem[Kat13b]{Kat13}
Nicholas~M. Katz.
\newblock Witt vectors and a question of {K}eating and {R}udnick.
\newblock {\em Int. Math. Res. Not. IMRN}, 2013(16):3613--3638, 2013.

\bibitem[Kat17]{Katz16}
Nicholas~M. Katz.
\newblock Witt vectors and a question of {R}udnick and {W}axman.
\newblock {\em Int. Math. Res. Not. IMRN}, 2017(11):3377--3412, 2017.

\bibitem[Kow06]{KowLS06}
Emmanuel Kowalski.
\newblock The large sieve, monodromy and zeta functions of curves.
\newblock {\em J. Reine Angew. Math.}, 601:29--69, 2006.

\bibitem[Kow08a]{KowLargeSieve08}
Emmanuel Kowalski.
\newblock {\em The large sieve and its applications}, volume 175 of {\em
  Cambridge Tracts in Mathematics}.
\newblock Cambridge University Press, Cambridge, 2008.
\newblock Arithmetic geometry, random walks and discrete groups.

\bibitem[Kow08b]{Kow08}
Emmanuel Kowalski.
\newblock The large sieve, monodromy, and zeta functions of algebraic curves.
  {II}. {I}ndependence of the zeros.
\newblock {\em Int. Math. Res. Not. IMRN}, rnn091, 2008.

\bibitem[KR14]{KeatRud14}
Jonathan~P. Keating and Ze\'{e}v Rudnick.
\newblock The variance of the number of prime polynomials in short intervals
  and in residue classes.
\newblock {\em Int. Math. Res. Not. IMRN}, 2014(1):259--288, 2014.

\bibitem[KS99]{KatzSarnak91}
Nicholas~M. Katz and Peter Sarnak.
\newblock {\em Random matrices, {F}robenius eigenvalues, and monodromy},
  volume~45 of {\em American Mathematical Society Colloquium Publications}.
\newblock American Mathematical Society, Providence, RI, 1999.

\bibitem[KS10]{KerSchm10}
Moritz Kerz and Alexander Schmidt.
\newblock On different notions of tameness in arithmetic geometry.
\newblock {\em Math. Ann.}, 346(3):641--668, 2010.

\bibitem[Lar95]{Lars95}
Michael Larsen.
\newblock Maximality of {Galois} actions for compatible systems.
\newblock {\em Duke Math. J.}, 80(3):601--630, 1995.

\bibitem[Li18]{Li18}
Wanlin Li.
\newblock Vanishing of hyperelliptic {$L$}-functions at the central point.
\newblock {\em J. Number Theory}, 191:85--103, 2018.

\bibitem[Liv87]{Liv87}
Ron Livn\'{e}.
\newblock The average distribution of cubic exponential sums.
\newblock {\em J. Reine Angew. Math.}, 375/376:362--379, 1987.

\bibitem[LP92]{LarsPink92}
Michael Larsen and Richard Pink.
\newblock On {$\ell$}-independence of algebraic monodromy groups in compatible
  systems of representations.
\newblock {\em Invent. Math.}, 107(1):603--636, 12 1992.

\bibitem[MN17]{MarNg17}
Greg Martin and Nathan Ng.
\newblock Inclusive prime number races.
\newblock 2017.
\newblock Preprint arXiv:1710.00088.

\bibitem[MT11]{TesMal11}
Gunther Malle and Donna Testerman.
\newblock {\em Linear algebraic groups and finite groups of {Lie} type}, volume
  133 of {\em Cambridge {Studies} in {Advanced} {Mathematics}}.
\newblock Cambridge University Press, 2011.

\bibitem[MVW84]{MVW84}
Charles~R. Matthews, Leonid~N. Vaserstein, and Boris Weisfeiler.
\newblock Congruence properties of {Z}ariski-dense subgroups. {I}.
\newblock {\em Proc. London Math. Soc. (3)}, 48(3):514--532, 1984.

\bibitem[Nar04]{Nark04}
Władysław Narkiewicz.
\newblock {\em Elementary and analytic theory of algebraic numbers}.
\newblock Springer Monographs in Mathematics. Springer-Verlag, Berlin, third
  edition, 2004.

\bibitem[PG17]{PG16}
Corentin Perret-Gentil.
\newblock Gaussian distribution of short sums of trace functions over finite
  fields.
\newblock {\em Math. Proc. Cambridge Philos. Soc.}, 163(3):385--422, 2017.

\bibitem[PG18a]{PG18}
Corentin Perret-Gentil.
\newblock Exponential sums over finite fields and the large sieve.
\newblock {\em Int. Math. Res. Not. IMRN}, rny202, 2018.

\bibitem[PG18b]{PGIntMonKS16}
Corentin Perret-Gentil.
\newblock Integral monodromy groups of {K}loosterman sheaves.
\newblock {\em Mathematika}, 64(3):652--678, 6 2018.

\bibitem[Pin00]{Pink00}
Richard Pink.
\newblock Strong approximation for {Z}ariski dense subgroups over arbitrary
  global fields.
\newblock {\em Comment. Math. Helv.}, 75(4):608--643, 2000.

\bibitem[Rai97]{Rains97}
Eric~M. Rains.
\newblock High powers of random elements of compact {L}ie groups.
\newblock {\em Probab. Theory Related Fields}, 107(2):219--241, 1997.

\bibitem[Ros02]{Ros02}
Michael Rosen.
\newblock {\em Number theory in function fields}, volume 210 of {\em Graduate
  Texts in Mathematics}.
\newblock Springer-Verlag, New York, 2002.

\bibitem[RS94]{RubSar94}
Michael Rubinstein and Peter Sarnak.
\newblock Chebyshev's bias.
\newblock {\em Experiment. Math.}, 3(3):173--197, 1994.

\bibitem[RW18]{RudWax17}
Zeév Rudnick and Ezra Waxman.
\newblock Angles of {Gaussian} primes.
\newblock {\em Israel J. Math.}, 232(1), 2018.

\bibitem[Spe11]{MOSpe11}
David~E. Speyer.
\newblock Faithful representations and tensor powers.
\newblock MathOverflow \url{https://mathoverflow.net/questions/18194/}, 4 2011.

\bibitem[Ste62]{Ste62}
Robert Steinberg.
\newblock Complete sets of representations of algebras.
\newblock {\em Proc. Amer. Math. Soc.}, 13:746--747, 1962.

\bibitem[Ste93]{Ste93}
Elias~M. Stein.
\newblock {\em Harmonic analysis: real-variable methods, orthogonality, and
  oscillatory integrals}, volume~43 of {\em Princeton Mathematical Series}.
\newblock Princeton University Press, Princeton, NJ, 1993.
\newblock With the assistance of Timothy S. Murphy, Monographs in Harmonic
  Analysis, III.

\bibitem[vdPS91]{vdpS91}
Alfred~J. van~der Poorten and Hans~Peter Schlickewei.
\newblock Additive relations in fields.
\newblock {\em J. Austral. Math. Soc. Ser. A}, 51(1):154--170, 1991.

\bibitem[Wei84]{Weis84}
Boris Weisfeiler.
\newblock Strong approximation for {Z}ariski-dense subgroups of semisimple
  algebraic groups.
\newblock {\em Ann. of Math. (2)}, 120(2):271--315, 1984.

\end{thebibliography}
\else
\printbibliography
\fi

\end{document}